\documentclass[11pt,a4paper,oldfontcommands]{article}
\usepackage[utf8]{inputenc}
\usepackage[T1]{fontenc}
\usepackage{xcolor}
\usepackage{times}

\usepackage[english]{babel}
\usepackage{envmath}
\usepackage{mathtools}
\usepackage{amsfonts}
\usepackage{mathrsfs}
\usepackage{amsthm}
\usepackage{subfigure}
\usepackage{amssymb}
\usepackage{dsfont}
\usepackage{bigints}
\usepackage{textgreek}
\usepackage{authblk}
\usepackage{titlesec}
\usepackage{enumitem}
\usepackage{url}

\usepackage[
breaklinks=true,
colorlinks=false,
bookmarks=true,bookmarksopenlevel=2]{hyperref}

\usepackage{geometry}
\newtheorem{theorem}{Theorem}[section]

\newtheorem{lem}[theorem]{Lemma}
\newtheorem{cor}[theorem]{Corollary}

\newtheorem{assump}{Assumption}
\theoremstyle{remark}
\newtheorem{rem}[theorem]{Remark}

\numberwithin{equation}{section}

\renewcommand{\Re}{\operatorname{Re}}

\renewcommand{\Im}{\operatorname{Im}}

\newcommand{\divg}{\operatorname{div}}

\newcommand*\diff{\mathop{}\!\mathrm{d}}

\DeclarePairedDelimiter\abs{\lvert}{\rvert}%
\DeclarePairedDelimiter\norm{\lVert}{\rVert}%
\newcommand{\tnorm}[1]{{\left\vert\kern-0.25ex\left\vert\kern-0.25ex\left\vert #1 
    \right\vert\kern-0.25ex\right\vert\kern-0.25ex\right\vert}}

\makeatletter
\let\oldabs\abs
\def\abs{\@ifstar{\oldabs}{\oldabs*}}
\let\oldnorm\norm
\def\norm{\@ifstar{\oldnorm}{\oldnorm*}}
\makeatother

\newcommand*{\myemail}[1]{%
    \normalsize\href{mailto:#1}{#1}\par
    }

\allowdisplaybreaks

\titleformat{\section}[block]{\centering \scshape \large}{\thesection.}{0.3\baselineskip}{}
\titlespacing{\section}{0pt}{*5}{*2}

\titleformat{\subsection}[block]{\bfseries}{\thesubsection.}{.5em}{}
\titlespacing{\subsection}{0pt}{*2.5}{*1}

\titleformat{\subsubsection}[runin]{\itshape}{\normalfont \thesubsubsection.}{.5em}{}[.]
\titlespacing{\subsubsection}{0pt}{*2.5}{0.5em}

\title{WKB analysis of the Logarithmic Nonlinear Schrodinger Equation in an analytic framework.}
\date{\vspace{-1cm}}

\author[]{Guillaume Ferriere}

\affil[]{IMAG, Univ Montpellier, CNRS, Montpellier, France \\ \myemail{guillaume.ferriere@umontpellier.fr}}

\begin{document}

\maketitle

\begin{abstract}
    We are interested in a WKB analysis of the Logarithmic Non-Linear Schrödinger Equation with "Riemann-like" variables in an analytic framework in semiclassical regime.
    We show that the Cauchy problem is locally well posed uniformly in the semiclassical constant and that the semiclassical limit can be performed.
    In particular, our framework is not only compatible with the Gross-Pitaevskii equation with logarithmic nonlinearity, but also allows initial data (and solutions) which can converge to $0$ at infinity.
\end{abstract}

\section{Introduction}

\subsection{Setting}

We are interested in the \textit{Logarithmic Non-Linear Schrödinger Equation} (also called logNLS)
\begin{equation}
    i \varepsilon \partial_t u^\varepsilon + \frac{\varepsilon^2}{2} \Delta u^\varepsilon = \lambda \ln{\abs{u^\varepsilon}^2} u^\varepsilon, \qquad \qquad u^\varepsilon (0) = u^\varepsilon_\textnormal{in}, \label{eq:lognls}
\end{equation}
with $x \in \mathbb{R}^d$, $d \geq 1$, $\lambda \in \mathbb{R} \setminus \{ 0 \}$, $\varepsilon > 0$. This equation was introduced as a model of nonlinear wave mechanics and in nonlinear optics (\cite{nonlin_wave_mec}, see also \cite{inco_white_light_log, log_nls_nuclear_physics, quantal_damped_motion, solitons_log_med, log_nls_magma_transp}).
The case $\lambda > 0$ (whose study of the Cauchy problem goes back to \cite{cazenave-haraux, Guerrero_Lopez_Nieto_H1_solv_lognls}) was studied by R. Carles and I. Gallagher who made explicit an unusually faster dispersion with a universal behaviour of the modulus of the solution (see \cite{carlesgallagher}). The knowledge of this behaviour was recently improved with a convergence rate but also extended through the semiclassical limit in \cite{Ferriere__Wass_semiclass_defoc_NLS}.
On the other hand, the case $\lambda < 0$ seems to be the most interesting from a physical point of view and has been studied formally and rigorously (see for instance \cite{Cazenave_log_nls, Dav_Mont_Squa_lognls, quantal_damped_motion, carlesnouri, Ferriere__superposition_logNLS, Ferriere__existence_multi_solitons_logNLS}).

This paper addresses the semiclassical limit of \eqref{eq:lognls} for general $\lambda \neq 0$ through WKB analysis in an analytic framework. For this, we first address the Cauchy problem of the system given by this WKB analysis (see \eqref{sys:euler_modified_semicla}) and give a local Cauchy theory independent of $\varepsilon \in [0,1]$. Then, we prove that the solutions for $\varepsilon > 0$ converge when $\varepsilon \rightarrow 0$ to the solution constructed for $\varepsilon = 0$ as expected.
Last, we address the complete convergence of the wave function $u^\varepsilon$ as $\varepsilon \rightarrow 0$.

\subsection{The WKB analysis for NLS}

In the case $\lambda > 0$, R. Carles and A. Nouri \cite{carlesnouri} have performed a WKB analysis of this equation: for initial data of the form $u^\varepsilon_\textnormal{in} = \sqrt{\rho_\textnormal{in}} \, e^{ i \frac{\phi_\textnormal{in}}{\varepsilon}}$ (in general dimension $d$), one can seek $u^\varepsilon$ under the form $u^\varepsilon = a^\varepsilon e^{i \frac{\phi^\varepsilon}{\varepsilon}}$ where $a^\varepsilon \in \mathbb{C}$ and $\phi^\varepsilon \in \mathbb{R}$ satisfy:
\begin{System} \label{sys:wkb}
    \partial_t \phi^\varepsilon + \frac{1}{2} \nabla \phi^\varepsilon \cdot \nabla \phi^\varepsilon + \lambda \ln{\abs{a^\varepsilon}^2} = 0, \qquad \qquad &\phi^\varepsilon (0) = \phi_\textnormal{in}, \\
    \partial_t a^\varepsilon + \nabla \phi^\varepsilon \cdot \nabla a^\varepsilon + \frac{1}{2} a^\varepsilon \Delta \phi^\varepsilon = i \frac{\varepsilon}{2} \Delta a^\varepsilon, \qquad \qquad  &a^\varepsilon (0) = \sqrt{\rho_\textnormal{in}}.
\end{System}
Note that allowing $a^\varepsilon$ to be complex-valued (even though $\sqrt{\rho_\textnormal{in}}$ is real-valued) gives a degree of freedom to dispatch terms from \eqref{eq:lognls} into this system. Then, they follow the choice introduced by Grenier which is more robust than the Madelung transform when semiclassical limit is considered (see \cite{Carles_Danchin_Saut__Madelung_GP_Kort}).
From this system, one usually defines
\begin{equation} \label{eq:rel_v_phi}
    v^\varepsilon \coloneqq \nabla \phi^\varepsilon.
\end{equation}
This relation is also equivalent to (see \cite{carlesnouri})
\begin{equation*}
    \phi^\varepsilon (t,x) = \phi_\textnormal{in} (x) - \int_0^t \Bigl( \frac{1}{2} \abs{v^\varepsilon (\tau,x)}^2 + \lambda \ln{\abs{a^\varepsilon (\tau,x)}^2} \Bigr) \diff \tau, \label{eq:rel_v_phi_2}
\end{equation*}
so that, along with $a^\varepsilon$, determining $\phi^\varepsilon$ turns out to be equivalent to determining $v^\varepsilon$ solution to
\begin{System} \label{eq:1st_rel}
    \partial_t v^\varepsilon + ( v^\varepsilon \cdot \nabla ) v^\varepsilon + \lambda \, \nabla \Bigl( \ln{\abs{a^\varepsilon}^2} \Bigr) = 0, \qquad \qquad &v^\varepsilon (0) = \nabla \phi_\textnormal{in}, \\
    \partial_t a^\varepsilon + v^\varepsilon \cdot \nabla a^\varepsilon + \frac{1}{2} a^\varepsilon \divg v^\varepsilon = i \frac{\varepsilon}{2} \Delta a^\varepsilon, \qquad \qquad  &a^\varepsilon (0) = \sqrt{\rho_\textnormal{in}}.
\end{System}

\begin{rem}
    To get this system, we have used the fact that
    \begin{equation} \label{eq:conv_term_v_k}
        \frac{1}{2} \nabla ( \abs{v^\varepsilon}^2 ) = (v^\varepsilon \cdot \nabla) v^\varepsilon,
    \end{equation}
    due to the fact that $v^\varepsilon$ is a gradient. All across this paper, we will use the general fact that, for an irrotational field $f$, one has
    \begin{equation*}
        \frac{1}{2} \nabla ( \abs{f}^2 ) = (f \cdot \nabla) f.
    \end{equation*}
\end{rem}

The semiclassical limit $\varepsilon \rightarrow 0$ for $u^\varepsilon$ relates classical and quantum wave equations and is expected to be described by the laws of hydrodynamics (see e.g. \cite{Gerard__Semicla_NLS, Grenier_semicla_NLS, Gasser_Lin_Markowich__dispersive_lim_NLS, Desjardins_Lin__semicla_mod_NLS}).
In particular, passing formally to the limit $\varepsilon \rightarrow 0$ in \eqref{eq:1st_rel} leads to:
\begin{System} \label{eq:sym_iso_euler}
    \partial_t v + (v \cdot \nabla) v + \lambda \, \nabla \Bigl( \ln{\abs{a}^2} \Bigr) = 0, \qquad \qquad &v (0) = \nabla \phi_\textnormal{in}, \\
    \partial_t a + v \cdot \nabla a + \frac{1}{2} a \divg v = 0, \qquad \qquad  &a (0) = \sqrt{\rho_\textnormal{in}}.
\end{System}
which is the symmetrized version of the isothermal Euler system ($\varepsilon = 0$) with $\rho = \abs{a}^2$ (see \cite{Chemin_Dynamique_gaz, Makino_Ukai_Kawashima__Euler}):
\begin{System} \label{eq:iso_euler}
    \partial_t \rho + \divg{(\rho v)} = 0, \\
    \partial_t (\rho v) + \divg{(\rho v \otimes v)} + \lambda \nabla \rho = 0.
\end{System}

The WKB analysis is not exclusive to \eqref{eq:lognls}, it has been used a lot for general non-linear Schrödinger equations. For instance, it has been discussed in \cite{Carles_WKB_NLS} for general nonlinearity of the form
\begin{equation} \label{eq:gen_nls}
    i \varepsilon \partial_t u^\varepsilon + \frac{\varepsilon^2}{2} \Delta u^\varepsilon = \varepsilon^\kappa f({\abs{u^\varepsilon}^2}) u^\varepsilon.
\end{equation}
In particular, the WKB type analysis is justified for $\kappa \geq 1$, which corresponds to a weak nonlinearity. When $\kappa = 0$, the mathematical analysis of the semiclassical limit for nonlinear Schrodinger equations has been well developed for two cases: for analytic initial data (see for instance \cite{Gerard__Semicla_NLS, Thomann_Instability_NLS, Thomann_anal_NLS_manifolds}) and for initial data in some Sobolev space with a defocusing nonlinearity so that the analogue of \eqref{eq:1st_rel} is hyperbolic symmetric, possibly with a change of variables (see \cite{Grenier_semicla_NLS, Alazard_Carles__WKB_Sobolev, Chiron_Rousset__Geo_optics_NLS}).
It was also extended to the case of generalized derivative nonlinear Schrodinger equations (in dimension $d=1$)
\begin{equation*}
    i \varepsilon \partial_t u^\varepsilon + \frac{\varepsilon^2}{2} \Delta u^\varepsilon + i \frac{\varepsilon}{2} \partial_x ( g(\abs{u^\varepsilon}^2) u^\varepsilon ) = \varepsilon^\kappa f({\abs{u^\varepsilon}^2}) u^\varepsilon.
\end{equation*}
In \cite{Desjardins_Lin_Tso__Semicla_dNLS}, the semiclassical analysis for this equation relies on the assumption
\begin{equation*}
    \partial_x \phi^\varepsilon \, g' > 0, \qquad f \equiv 0,
\end{equation*}
and was generalized by \cite{Desjardins_Lin__semicla_mod_NLS} to the case
\begin{equation*}
    \partial_x \phi^\varepsilon \, g' + f' > 0.
\end{equation*}
These assumptions are made to ensure hyperbolicity, but have the strong drawback to involve the solution itself. However, hyperbolicity is not needed when one works with analytic functions (\cite{carles-gallo}). In this context, the semiclassical limit for \eqref{eq:gen_nls} (with $\kappa = 0$) was studied by \cite{Gerard__Semicla_NLS, Thomann_anal_NLS_manifolds}, thanks to some tools developed by J. Sjöstrand \cite{Sjostrand__Sing_anal_microloc}.

On the other hand, WKB analysis is also useful for the study of the Gross-Pitaevskii equation, for instance in the context where initial data do not necessarily decay to zero at infinity. The Cauchy problem (\cite{Gallo_Schr_group_Zhidkov, Gerard_Cauchy_GP}) and the semiclassical limit (\cite{Lin_Zhang__Semicla_GP}) of \eqref{eq:gen_nls} with $f(y) = y - 1$ and $\kappa = 0$ for instance have already been studied. In this case, the Hamiltonian structure yields that
\begin{equation*}
    \mathcal{E} (u^\varepsilon) = \varepsilon^2 \norm{\nabla u^\varepsilon (t)}_{L^2}^2 + \norm{\abs{u^\varepsilon (t)}^2 - 1}_{L^2}^2
\end{equation*}
is independent of time, at least formally, which leads to a natural energy space,
\begin{equation*}
    E = \{ u \in H^1_\textnormal{loc} ; \nabla u \in L^2, \abs{u}^2 - 1 \in L^2 \},
\end{equation*}
to study the Cauchy problem (see also \cite{Bethuel_Saut__Travel_waves_GP}). The modulus of functions in this space morally goes to $1$ at infinity. For more general initial data which are bounded but may not be in this space (for instance if they have several limits at infinity), P. E. Zhidkov introduced in the one-dimensional case in \cite{Zhidkov__Cauchy_NLS, Zhidkov_KdV_NLS} the so-called Zhidkov spaces:
\begin{equation*}
    X^s = \{ u \in L^\infty ; \nabla u \in H^{s-1} \}, \qquad s > \frac{d}{2}.
\end{equation*}
The study of theses spaces was generalized in the multidimensional case by C. Gallo (\cite{Gallo_Schr_group_Zhidkov}). They were also used by T. Alazard and R. Carles \cite{Alazard_Carles__WKB_GP_analytic} in their WKB analysis for the Gross-Pitaevskii equation. R. Carles and A. Nouri \cite{carlesnouri} have also shown that, for initial data in Zhidkov spaces bounded away from vacuum, the Wigner measure of $u^\varepsilon$ solution to \eqref{eq:lognls} weakly converges to a monokinetic measure $f = \rho (t,x) \otimes \delta_{\xi = v(t,x)}$ such that $(\rho, v)$ satisfies the isothermal Euler system, thanks to the WKB analysis described before.

\subsection{Riemann invariants}

The isothermal Euler system \eqref{eq:iso_euler} has been studied a lot in different contexts (for example \cite{Tsuge__iso_Euler_spher, Bhat_Fetecau_regul_iso_Euler, Dong__blowup_iso_Euler, Chen_Li__entropy_sol_iso_Euler, Banda_Herty_Klar__iso_Euler_network}).
B. Riemann solved the "Riemann problem" for this equation in his memoir to the Royal Academy of Sciences of Göttinger (1860) (see \cite{Serre_hyperbolic}).
In dimension $d=1$ and in the case $\lambda > 0$, he introduced the so-called Riemann invariants $w_1 = v + \sqrt{2 \lambda} \ln{\rho}$ and $w_2 = v - \sqrt{2 \lambda} \ln{\rho}$. Then, he proved that the necessary and sufficient condition for the solution to exist for all positive times is that $w_1$ (resp. $w_2$) is non-decreasing (resp. non-increasing).

This shows that the good unknown to be considered would rather be $\ln{\rho}$ in \eqref{eq:iso_euler}, or $\ln{a}$ in \eqref{eq:sym_iso_euler}. In particular, we should therefore consider $\ln{a^\varepsilon}$ in \eqref{eq:1st_rel}.
This intuition is strengthened by the fact that dividing the second equation by $a^\varepsilon$ in system \eqref{eq:1st_rel} gives
\begin{equation*}
    \frac{\partial_t a^\varepsilon}{a^\varepsilon} + v^\varepsilon \cdot \frac{\nabla a^\varepsilon}{a^\varepsilon} + \frac{1}{2} \divg v^\varepsilon = i \frac{\varepsilon}{2} \frac{\Delta a^\varepsilon}{a^\varepsilon},
\end{equation*}
which could formally be written in terms of $\ln{a^\varepsilon}$ only, since we also have (at least for $f$ real)
\begin{equation*}
    \frac{\Delta f}{f} = \Delta (\ln{f}) + \nabla (\ln{f}) \cdot \nabla (\ln{f}).
\end{equation*}
Moreover, in the first equation in \eqref{eq:1st_rel}, if formally $f = \ln{a^\varepsilon}$ \textit{i.e.} $a^\varepsilon = e^f$, then
\begin{equation*}
    \ln{\abs{a^\varepsilon}^2} = 2 \Re f.
\end{equation*}

However, in the latter, contrary to $a$, $a^\varepsilon$ is complex, a fact which may lead to some problems when defining $\ln{a^\varepsilon}$.
Still, we can override this difficulty. Indeed, $a^\varepsilon$ is defined this way only to get $u^\varepsilon$ solution to \eqref{eq:lognls}. Instead of defining $a^\varepsilon$, one can try to directly define $\psi^\varepsilon$ (along with $\phi^\varepsilon$) such that $u^\varepsilon = e^{\frac{\psi^\varepsilon}{2} + i \frac{\phi^\varepsilon}{\varepsilon}}$ is solution to \eqref{eq:lognls}.
For this, we first assume $\rho_\textnormal{in}^\varepsilon = e^{\psi_\textnormal{in}^\varepsilon}$, so that $u^\varepsilon_\textnormal{in} = e^{\frac{\psi_\textnormal{in}^\varepsilon}{2} + i \frac{\phi_\textnormal{in}^\varepsilon}{\varepsilon}}$ (note that we allow the initial data to depend on $\varepsilon$ with suitable conditions which will be made explicit later).
Thus we can seek $u^\varepsilon$ under the form $u^\varepsilon = e^{\frac{\psi^\varepsilon}{2} + i \frac{\phi^\varepsilon}{\varepsilon}}$ with $\psi^\varepsilon \in \mathbb{C}$ and $\phi^\varepsilon \in \mathbb{R}$ such that:
\begin{System} \label{sys:wkb_riemann}
    \partial_t \phi^\varepsilon + \frac{1}{2} \nabla \phi^\varepsilon \cdot \nabla \phi^\varepsilon + \lambda \Re \psi^\varepsilon = 0, \qquad \qquad &\phi^\varepsilon (0) = \phi_\textnormal{in}^\varepsilon, \\
    \partial_t \psi^\varepsilon + \nabla \phi^\varepsilon \cdot \nabla \psi^\varepsilon + \Delta \phi^\varepsilon = i \frac{\varepsilon}{2} \Bigl( \Delta \psi^\varepsilon + 2 \, \nabla \psi^\varepsilon \cdot \nabla \psi^\varepsilon \Bigr), \qquad \qquad  &\psi^\varepsilon (0) = \psi_\textnormal{in}^\varepsilon.
\end{System}
Note that the pseudo scalar product in the second equation is defined for $a, b \in \mathbb{C}^d$ by
\begin{equation*}
    a \cdot b = \sum_i a_i \, b_i,
\end{equation*}
and therefore $a \cdot a$ is not necessarily real.
In the same way as for passing from \eqref{sys:wkb} to \eqref{eq:1st_rel}, we can define $v^\varepsilon \coloneqq \nabla \phi^\varepsilon$ \eqref{eq:rel_v_phi}, which leads to
\begin{System} \label{sys:euler_riemann_semicla}
    \partial_t v^\varepsilon + (v^\varepsilon \cdot \nabla) v^\varepsilon + \lambda \, \nabla \Bigl( \Re \psi^\varepsilon \Bigr) = 0,\qquad \qquad &v^\varepsilon (0) = \nabla \phi_\textnormal{in}^\varepsilon, \\
    \partial_t \psi^\varepsilon + v^\varepsilon \cdot \nabla \psi^\varepsilon + \divg v^\varepsilon = i \frac{\varepsilon}{2} \Bigl( \Delta \psi^\varepsilon + 2 \, \nabla \psi^\varepsilon \cdot \nabla \psi^\varepsilon \Bigr), \qquad \qquad  &\psi^\varepsilon (0) = \psi_\textnormal{in}^\varepsilon.
\end{System}
Moreover, \eqref{eq:rel_v_phi} is here equivalent to
\begin{equation} \label{eq:rel_v_phi_3}
    \phi^\varepsilon (t,x) = \phi_\textnormal{in}^\varepsilon (x) - \int_0^t \Bigl( \frac{1}{2} \abs{v^\varepsilon (\tau,x)}^2 + \lambda \Re \psi^\varepsilon (\tau, x) \Bigr) \diff \tau.
\end{equation}
As soon as the initial data converges, passing formally to the semiclassical limit $\varepsilon \rightarrow 0$ in system \eqref{sys:euler_riemann_semicla} yields:
\begin{System} \label{sys:euler_riemann}
    \partial_t v + (v \cdot \nabla) v + \lambda \, \nabla \psi = 0, \qquad \qquad &v (0) = \nabla \phi_\textnormal{in}, \\
    \partial_t \psi + v \cdot \nabla \psi + \divg v = 0, \qquad \qquad  &\psi (0) = \psi_\textnormal{in}.
\end{System}
This system is linked to the isothermal Euler system \eqref{eq:iso_euler} (and thus also to \eqref{eq:sym_iso_euler}) in the sense that it is the same but written in "Riemann-like" variables (with $\rho = e^{\psi}$).

\begin{rem}
    If we neglect the convective terms $(v \cdot \nabla) v$ and $v \cdot \nabla \psi$, which are still not regularizing terms and may already lead to shocks (like for the Burgers equation), then we would get the system
    \begin{System}
        \partial_t v + \lambda \, \nabla \psi = 0, \\
        \partial_t \psi + \divg v = 0.
    \end{System}
    For instance, $\psi$ would satisfy
    \begin{equation*}
        \partial^2_{tt} \psi - \lambda \Delta \psi = 0,
    \end{equation*}
    and a similar equation would hold for $v$.
    In particular, when $\lambda > 0$, we get the wave equation, which is well posed in $L^2$-based spaces (\textit{e.g.} $(\psi (0), \partial_t \psi (0)) \in H^1 \times L^2$), or based on Zhidkov spaces for instance.
    However, if $\lambda < 0$, this equation becomes way more singular. Indeed, for the Fourier transform, defined for every $f \in L^1$ and for every $\xi \in \mathbb{R}^d$ by
    \begin{equation*}
        \hat{f} (\xi) = \mathcal{F} (f) (\xi) \coloneqq \int f(x) e^{- i x \cdot \xi} \diff x,
    \end{equation*}
    and then extended for any $f \in L^2$,
    we get
    \begin{equation*}
        \partial^2_{tt} \hat{\psi} - \abs{\lambda} \abs{\xi}^2 \hat{\psi} = 0,
    \end{equation*}
    whose solutions are
    \begin{equation*}
        \hat{\psi} (\xi) = \hat{\psi} (0) \, \cosh ( \sqrt{\abs{\lambda}} \abs{\xi} t ) + \partial_t \hat{\psi} (0) \, \sinh ( \sqrt{\abs{\lambda}} \abs{\xi} t ).
    \end{equation*}
    Thus the Fourier transform is in $L^2$ for some interval $[0, T]$ only if the initial data are analytic.
    Hence, one may probably not hope for a Cauchy theory of \eqref{sys:euler_riemann} (and thus for \eqref{sys:euler_riemann_semicla}) in lower regularities for this case (see for instance \cite{Metivier__well_posed_Cauchy} in 1D and \cite{Lerner_Nguyen_Texier__instability_1st_order} in higher dimension). Our construction will still work for $\lambda > 0$, so we take the general case $\lambda \neq 0$.
\end{rem}

\subsection{Transformation of the system} \label{subsec:sys_v_zeta}

Obviously, we have
\begin{equation*}
    \nabla \Bigl( \Re \psi^\varepsilon \Bigr) = \Re \Bigl( \nabla \psi^\varepsilon \Bigr)
\end{equation*}
Therefore, the system \eqref{sys:euler_riemann_semicla} does not involve $\psi^\varepsilon$ directly, but only derivatives of this function. Therefore, in the same way as for $\phi^\varepsilon$ and $v^\varepsilon$, one can transform the equation for $\psi^\varepsilon$ into an equation for
\begin{equation} \label{eq:rel_zeta_psi}
    \zeta^\varepsilon \coloneqq \nabla \psi^\varepsilon,
\end{equation}
so that \eqref{sys:euler_riemann_semicla} becomes
\begin{System} \label{sys:euler_modified_semicla}
    \partial_t v^\varepsilon + (v^\varepsilon \cdot \nabla) v^\varepsilon + \lambda \, \Re \zeta^\varepsilon = 0,\qquad \qquad &v^\varepsilon (0) = \nabla \phi_\textnormal{in}^\varepsilon, \\
    \partial_t \zeta^\varepsilon + \nabla \Bigl( v^\varepsilon \cdot \zeta^\varepsilon \Bigr) + \nabla \divg v^\varepsilon = i \frac{\varepsilon}{2} \Bigl( \nabla \divg \zeta^\varepsilon + 2 \, \nabla ( \zeta^\varepsilon \cdot \zeta^\varepsilon ) \Bigr), \qquad \qquad  &\zeta^\varepsilon (0) = \nabla \psi_\textnormal{in}^\varepsilon.
\end{System}
In the same way as the relation between \eqref{eq:rel_v_phi} and \eqref{eq:rel_v_phi_3}, \eqref{eq:rel_zeta_psi} is equivalent to
\begin{equation} \label{eq:rel_psi_zeta}
    \psi^\varepsilon (t,x) = \psi_\textnormal{in}^\varepsilon (x) - \int_0^t \biggl[ v^\varepsilon (\tau, x) \cdot \zeta^\varepsilon (\tau, x) + \divg v^\varepsilon (\tau, x) - i \frac{\varepsilon}{2} \Bigl( \divg \zeta^\varepsilon (\tau, x) + 2 \, \zeta^\varepsilon (\tau, x) \cdot \zeta^\varepsilon (\tau, x) \Bigr) \biggr] \diff \tau.
\end{equation}
Indeed, it is obvious that, with the definition \eqref{eq:rel_psi_zeta}, $\psi^\varepsilon (0) = \psi_\textnormal{in}^\varepsilon$, \textit{i.e.} $\nabla \psi^\varepsilon (0) = \nabla \psi_\textnormal{in}^\varepsilon = \zeta^\varepsilon (0)$, and one can easily compute that
\begin{equation*}
    \partial_t (\nabla \psi^\varepsilon) - \partial_t \zeta^\varepsilon = \nabla \partial_t \psi^\varepsilon - \partial_t \zeta^\varepsilon = 0.
\end{equation*}

\subsection{Main results}

\subsubsection{Notations for analytic spaces} \label{subsubsec:analytic}

All across the paper, we denote
\begin{equation*}
    \langle \xi \rangle \coloneqq \sqrt{1 + \abs{\xi}^2}.
\end{equation*}
Then, for $\delta, \ell \geq 0$ and $n \in \mathbb{N}^*$, define the analytic spaces (like in \cite{Ginibre_Velo__Gevrey}):
\begin{equation*}
    \mathcal{H}_{\delta}^\ell ( \mathbb{R}^d, \mathbb{R}^n ) \coloneqq \{ f \in L^2 ( \mathbb{R}^d, \mathbb{R}^n ), \norm{f}_{\mathcal{H}_{\delta}^\ell} < \infty \},
\end{equation*}
where $L^2 ( \mathbb{R}^d, \mathbb{R}^n )$ designates the functions in $L^2$ from $\mathbb{R}^d$ with values in $\mathbb{R}^n$ and
\begin{equation*}
    \norm{f}_{\mathcal{H}_{\delta}^\ell}^2 = \int_{\mathbb{R}^d} \langle \xi \rangle^{2 \ell} e^{2 \delta \langle \xi \rangle} \abs{\hat{f} (\xi)}^2 \diff \xi \eqqcolon \norm{f}_{\ell, \delta}^2,
\end{equation*}
and $\hat{f} = \mathcal{F}(f)$ designates the Fourier transform in space variables. We also define the scalar product associated to this norm:
\begin{equation*}
    \langle f, g \rangle_{\ell, \delta} \coloneqq \int_{\mathbb{R}^d} \langle \xi \rangle^{2 \ell} e^{2 \delta \langle \xi \rangle} \hat{f} (\xi) \cdot \overline{\hat{g} (\xi)} \diff \xi.
\end{equation*}
For simplicity of notations, we will drop $( \mathbb{R}^d, \mathbb{R}^n )$ in the definition of $\mathcal{H}_\delta^\ell$ and also in $L^2$.
When we consider "continuous" $\mathcal{H}_{\delta}^\ell$ valued functions, these are functions that belong to
\begin{equation*}
    \mathcal{C} (I, \mathcal{H}_{\delta}^\ell) \coloneqq \{ f \in \mathcal{C} (I, L^2), \mathcal{F}^{-1} (w_\delta \hat{f}) \in \mathcal{C} (I, \mathcal{H}_0^{\ell}) = \mathcal{C} (I, H^\ell) \},
\end{equation*}
for some interval $I$ and where
\begin{equation*}
    w_\delta \coloneqq \exp{\Bigl( \delta \langle \xi \rangle \Bigr)},
\end{equation*}
with $\delta = \delta (t)$ continuous (and even $\mathcal{C}^1$).
When $I = [0, T]$, we denote
\begin{gather*}
    \mathcal{C}_T \mathcal{H}_{\delta}^\ell \coloneqq \mathcal{C} ([0, T], \mathcal{H}_{\delta}^\ell), \\
    L^\infty_T \mathcal{H}_{\delta}^\ell \coloneqq L^\infty ((0, T), \mathcal{H}_{\delta}^\ell), \\
    L^2_T \mathcal{H}_{\delta}^\ell \coloneqq L^2 ((0, T), \mathcal{H}_{\delta}^\ell), \\
    H^1_T \mathcal{H}_{\delta}^\ell \coloneqq H^1 ((0, T), \mathcal{H}_{\delta}^\ell).
\end{gather*}
Moreover, we also denote the following norms for any $t \in [0, T]$ as
\begin{gather*}
    \tnorm{f}_{\infty, t, \ell, \delta} \coloneqq \norm{f}_{L^\infty_t \mathcal{H}_{\delta}^\ell} = \sup_{\tau \in (0, t)} \norm{f (\tau)}_{\ell, \delta (\tau)}, \\
    \tnorm{f}_{2, t, \ell, \delta} \coloneqq \norm{f}_{L^2_t \mathcal{H}_{\delta}^\ell} = \Bigl( \int_0^t \norm{f (\tau)}_{\ell, \delta (\tau)}^2 \diff \tau \Bigr)^\frac{1}{2}.
\end{gather*}
Finally, for any $M, \ell, \delta > 0$ and $f \in L^\infty_T \mathcal{H}^\ell_\delta \cap L^2_T \mathcal{H}^{\ell + \frac{1}{2}}_\delta \cap \mathcal{C}_T \mathcal{H}^{\ell - \frac{1}{2}}_\delta$, we also define for all $t \in [0, T]$:
\begin{equation*}
    \mathcal{E}_{M, \ell, \delta} (f) (t) \coloneqq \norm{f (t)}_{\ell, \delta (t)}^2 + 2 M \, \tnorm{f}_{2, t, \ell, \delta}^2.
\end{equation*}

\subsubsection{Main result on $(\zeta^\varepsilon, v^\varepsilon)$}

We are interested in system \eqref{sys:euler_modified_semicla} in an analytic framework. For this, fix $\lambda \neq 0$, $\ell > \frac{d}{2}$ and $\delta_\textnormal{in} > 0$ for the rest of this paper. Then, we assume the following:
\begin{assump} \label{ass:bound}
    $\psi_\textnormal{in}^\varepsilon, \phi_\textnormal{in}^\varepsilon \in \mathcal{C}^1 (\mathbb{R}^d)$ are such that $\nabla \psi_\textnormal{in}^\varepsilon \in \mathcal{H}_{\delta_\textnormal{in}}^\ell$ and $\nabla \phi_\textnormal{in}^\varepsilon \in \mathcal{H}_{\delta_\textnormal{in}}^{\ell+1}$ are uniformly bounded in $\varepsilon \in [0, 1]$ in these spaces: there exists $\omega_\textit{in}$ such that for all $\varepsilon \in [0,1]$
    \begin{equation*}
        \norm{\nabla \psi_\textnormal{in}^\varepsilon}_{\ell, \delta_\textnormal{in}}^2 + \norm{\nabla \phi_\textnormal{in}^\varepsilon}_{\ell + 1, \delta_\textnormal{in}}^2 \leq \omega_\textnormal{in}.
    \end{equation*}
\end{assump}

\begin{rem}
    The initial data $\psi_\textnormal{in}^\varepsilon$ and $\phi_\textnormal{in}^\varepsilon$ might be unbounded when $\abs{x} \rightarrow \infty$, or have different limits at infinity (for instance at $\pm \infty$ in dim $d=1$, see Section \ref{sec:ass_in_data} and more specifically Lemma \ref{lem:ass_in_data}). In particular, this means that we allow the initial data $u^\varepsilon_\textnormal{in}$ for \eqref{eq:lognls} or $\rho_\textnormal{in}^0$ for \eqref{eq:iso_euler} to be near vacuum at infinity. This is different from \cite{carlesnouri}, which requires the initial data to be bounded away from zero (along with $\lambda > 0$).
\end{rem}

Our first main result is divided into two parts. Under the previous assumptions, the first part addresses the Cauchy problem of \eqref{sys:euler_modified_semicla}. With such a Cauchy theory, we then deal with semiclassical limit by stating that $(\zeta^\varepsilon, v^\varepsilon)$ converge (in some sense) to $(\zeta^0, v^0)$ as $\varepsilon \rightarrow 0$. For the latter,  we define for any $k > 0$ the following constant which depends only on $\varepsilon$:
\begin{equation*}
    (D^\varepsilon_k)^2 \coloneqq \norm{\nabla \psi^\varepsilon_\textnormal{in} - \nabla \psi^0_\textnormal{in}}_{k, \delta_\textnormal{in}}^2 + \norm{\nabla \phi^\varepsilon_\textnormal{in} - \nabla \phi^0_\textnormal{in}}_{k + 1, \delta_\textnormal{in}}^2.
\end{equation*}

\begin{theorem} \label{th:main_th}
    For any $(\psi_\textnormal{in}^\varepsilon, \phi_\textnormal{in}^\varepsilon)$ satisfying Assumption \ref{ass:bound}, there exists $T > 0$, $M > 0$ and $\delta = \delta (t) \coloneqq \delta_\textnormal{in} - M t$ such that, for all $\varepsilon \in [0,1]$:
    \begin{itemize}
    \item There exists a unique solution $(\zeta^\varepsilon, v^\varepsilon) \in L^\infty ((0, T), \mathcal{H}_{\delta}^\ell \times \mathcal{H}_{\delta}^{\ell+1}) \cap \mathcal{C} ([0, T], \mathcal{H}_{\delta}^{\ell - \frac{1}{2}} \times \mathcal{H}_{\delta}^{\ell + \frac{1}{2}}) \cap L^2 ((0, T), \mathcal{H}_{\delta}^{\ell+\frac{1}{2}} \times \mathcal{H}_{\delta}^{\ell+\frac{3}{2}})$ to \eqref{sys:euler_modified_semicla}.

    \item There exists $C > 0$ independent of $\varepsilon \in [0, 1]$ such that
    \begin{equation*}
        \tnorm{\zeta^\varepsilon - \zeta^0}_{\infty, T, \ell - \frac{1}{2}, \delta} + \tnorm{v^\varepsilon - v^0}_{\infty, T, \ell + \frac{1}{2}, \delta} + \tnorm{\zeta^\varepsilon - \zeta^0}_{2, T, \ell, \delta} + \tnorm{v^\varepsilon - v^0}_{2, T, \ell + 1, \delta} \leq C \Bigl( \sqrt{\varepsilon} + D^\varepsilon_{\ell - \frac{1}{2}} \Bigr),
    \end{equation*}
    and, if $\ell > \frac{d+1}{2}$,
    \begin{gather*}
        \tnorm{\zeta^\varepsilon - \zeta^0}_{\infty, T, \ell - 1, \delta} + \tnorm{v^\varepsilon - v^0}_{\infty, T, \ell, \delta} + \tnorm{\zeta^\varepsilon - \zeta^0}_{2, T, \ell - \frac{1}{2}, \delta} + \tnorm{v^\varepsilon - v^0}_{2, T, \ell  + \frac{1}{2}, \delta} \leq C \Bigl( {\varepsilon} + D^\varepsilon_{\ell - 1} \Bigr).
    \end{gather*}
    \end{itemize}
\end{theorem}

\subsubsection{Main result on $(\psi^\varepsilon, \phi^\varepsilon)$}

Once we have a Cauchy theory for the system \eqref{sys:euler_modified_semicla}, we can define $\psi^\varepsilon$ with \eqref{eq:rel_psi_zeta}. From this definition, it is easy to prove that $\nabla \psi^\varepsilon = \zeta^\varepsilon$ like in Section \ref{subsec:sys_v_zeta}.
In a similar way, we then define $\phi^\varepsilon$ with \eqref{eq:rel_v_phi_3}. However, we can not prove directly that $\nabla \phi^\varepsilon = v^\varepsilon$. Indeed, we need \eqref{eq:conv_term_v_k} to hold, \textit{i.e.} that $v^\varepsilon$ is irrotational. This is obviously true at $t = 0$. To prove it for $t > 0$, note that taking the curl of the equation on $v^\varepsilon$ gives a linear equation on $\operatorname{curl} v^\varepsilon$.
Therefore, from the previous result, we also gain a local Cauchy theory for \eqref{sys:wkb_riemann} through the relations \eqref{eq:rel_v_phi_3} and \eqref{eq:rel_psi_zeta}. Moreover, the semiclassical limit can also be extended to these functions, which leads to define for any $k > 0$ the following constant which depends only on $\varepsilon$:
\begin{equation} \label{eq:def_D_tilde}
    (\tilde{D}^\varepsilon_k)^2 \coloneqq \norm{\psi^\varepsilon_\textnormal{in} - \psi^0_\textnormal{in}}_{k, \delta_\textnormal{in}}^2 + \norm{\phi^\varepsilon_\textnormal{in} - \phi^0_\textnormal{in}}_{k + 1, \delta_\textnormal{in}}^2.
\end{equation}
Since the assumptions for the initial data may lead to non-trivial behavior at infinity for $(\psi^\varepsilon_\textnormal{in}, \phi^\varepsilon_\textnormal{in})$, we also address the behavior at infinity (in space) of $(\psi^\varepsilon (t), \phi^\varepsilon (t))$ thanks to the relations \eqref{eq:rel_v_phi_3} and \eqref{eq:rel_psi_zeta}.

\begin{cor} \label{cor:cauchy}
    For any $(\psi_\textnormal{in}^\varepsilon, \phi_\textnormal{in}^\varepsilon)$ satisfying Assumption \ref{ass:bound}, there exists $T > 0$, $M > 0$ and $\delta = \delta (t) \coloneqq \delta_\textnormal{in} - M t$ such that, for all $\varepsilon \in [0,1]$:
    \begin{itemize}
    \item There exists a unique solution $(\psi^\varepsilon, \phi^\varepsilon) \in \mathcal{C}^2 ([0, T] \times \mathbb{R}^d)^2$ to \eqref{sys:wkb_riemann} such that $(\nabla \psi^\varepsilon, \nabla \phi^\varepsilon) \in L^\infty ((0, T), \mathcal{H}_{\delta}^\ell \times \mathcal{H}_{\delta}^{\ell+1}) \cap \mathcal{C} ([0, T], \mathcal{H}_{\delta }^{\ell - \frac{1}{2}} \times \mathcal{H}_{\delta}^{\ell + \frac{1}{2}}) \cap L^2 ((0, T), \mathcal{H}_{\delta }^{\ell+\frac{1}{2}} \times \mathcal{H}_{\delta}^{\ell+\frac{3}{2}})$.

    \item There holds
    \begin{gather*}
        \psi^\varepsilon - \psi_\textnormal{in}^\varepsilon \in H^{1}_T \mathcal{H}_{\delta}^{\ell-\frac{1}{2}} \cap \mathcal{C}_T \mathcal{H}_{\delta}^{\ell+\frac{1}{2}} \cap L^\infty_T \mathcal{H}_{\delta}^{\ell+1} \cap L^2_T \mathcal{H}_{\delta}^{\ell+\frac{3}{2}}, \label{eq:1st_point_main_cor} \\
        \phi^\varepsilon - \phi_\textnormal{in}^\varepsilon - \lambda t \, \psi_\textnormal{in}^\varepsilon  \in H^{1}_T \mathcal{H}_{\delta}^{\ell+\frac{3}{2}} \cap L^\infty_T \mathcal{H}_{\delta}^{\ell+2} \cap L^2_T \mathcal{H}_{\delta}^{\ell+\frac{5}{2}}. \label{eq:2nd_point_main_cor}
    \end{gather*}
    
    \item There exists $C > 0$ independent of $\varepsilon \in [0, 1]$ such that
    \begin{multline*}
        \tnorm{\psi^\varepsilon - \psi^0}_{\infty, T, \ell + \frac{1}{2}, \delta} + \tnorm{\phi^\varepsilon - \phi^0}_{\infty, T, \ell + \frac{3}{2}, \delta} + \tnorm{\psi^\varepsilon - \psi^0}_{2, T, \ell + 1, \delta} + \tnorm{\phi^\varepsilon - \phi^0}_{2, T, \ell + 2, \delta} \\ \leq C \Bigl( \sqrt{\varepsilon} + \tilde{D}^\varepsilon_{\ell + \frac{1}{2}} \Bigr),
    \end{multline*}
    and, if $\ell > \frac{d+1}{2}$,
        \begin{gather*}
        \tnorm{\psi^\varepsilon - \psi^0}_{\infty, T, \ell, \delta} + \tnorm{\phi^\varepsilon - \phi^0}_{\infty, T, \ell + 1, \delta} + \tnorm{\psi^\varepsilon - \psi^0}_{2, T, \ell + \frac{1}{2}, \delta} + \tnorm{\phi^\varepsilon - \phi^0}_{2, T, \ell + \frac{3}{2}, \delta} \leq C \Bigl( \varepsilon + \tilde{D}^\varepsilon_\ell \Bigr).
    \end{gather*}

    \end{itemize}
    
\end{cor}

\begin{rem}
    $a^\varepsilon = e^{\frac{\psi^\varepsilon}{2}}$ satisfies \eqref{sys:wkb} and \eqref{eq:1st_rel} (along with $\phi^\varepsilon$ and $v^\varepsilon$ respectively). In particular, for $\varepsilon > 0$, $u^\varepsilon = e^{\frac{\psi^\varepsilon}{2} + i \frac{\phi^\varepsilon}{\varepsilon}}$ is a $\mathcal{C}^2$ solution to \eqref{eq:lognls} where $\phi^\varepsilon$ is defined by \eqref{eq:rel_v_phi_3}.
\end{rem}

\begin{rem} \label{rem:sc_eff}
    One can add any constant to $v^\varepsilon$, \textit{i.e.} any linear function to $\phi^\varepsilon$, thanks to the Galilean invariance: for any $c_0 \in \mathbb{R}^d$ and any $(\zeta^\varepsilon, v^\varepsilon)$ solution to \eqref{sys:euler_modified_semicla}, $(\zeta^\varepsilon (t, x - c_0 t), v^\varepsilon (t, x - c_0 t) + c_0)$ is also solution to \eqref{sys:euler_modified_semicla}, and a similar relation holds for $\psi^\varepsilon$ and $\phi^\varepsilon$.
    Moreover, the addition of a constant to $\psi_\textnormal{in}$ gives an explicit behaviour thanks to the effect of scaling for \eqref{eq:lognls}: if $u^\varepsilon$ is a solution to \eqref{eq:lognls} and $\kappa \in \mathbb{R}$, then
    \begin{equation*}
        u^\varepsilon(t,x) \, e^{ \kappa - \frac{2it \lambda \kappa}{\varepsilon}}
    \end{equation*}
    also solves \eqref{eq:lognls} (with initial datum $e^{\kappa} u^\varepsilon (0)$). The corresponding relation for $(\psi^\varepsilon, \phi^\varepsilon)$ is that, if $(\psi^\varepsilon, \phi^\varepsilon)$ satisfies \eqref{sys:wkb_riemann}, then $(\psi^\varepsilon + 2 \kappa, \phi^\varepsilon - 2 \lambda \kappa t)$ also satisfies \eqref{sys:wkb_riemann}, and this also holds for $\varepsilon = 0$.
    The second part of Corollary \ref{cor:cauchy}, in particular the term $\lambda t \, \psi_\textnormal{in}^\varepsilon$, is therefore consistent with the effect of the scaling.
\end{rem}

\begin{rem}
    In the case where $\psi_\textnormal{in}^\varepsilon$ and $\phi_\textnormal{in}^\varepsilon$ are independent of $\varepsilon$, the convergences are in $O (\sqrt{\varepsilon})$ and $O (\varepsilon)$ respectively for the semiclassical parts in Theorem \ref{th:main_th} and Corollary \ref{cor:cauchy}. Actually, if $\ell > \frac{d+1}{2}$, the first case can be deduced from the second case (in both the second part of Theorem \ref{th:main_th} and the third one of Corollary \ref{cor:cauchy}) by the fact that $(\zeta^\varepsilon, v^\varepsilon)$ is uniformly bounded (for $\varepsilon \in [0, 1]$) in $L^\infty ((0, T), \mathcal{H}_{\delta}^\ell \times \mathcal{H}_{\delta}^{\ell+1}) \cap L^2 ((0, T), \mathcal{H}_{\delta}^{\ell+\frac{1}{2}} \times \mathcal{H}_{\delta}^{\ell+\frac{3}{2}})$.
\end{rem}

The second part of Corollary \ref{cor:cauchy} gives useful information about the behavior of $\psi^\varepsilon (t)$ and $\phi^\varepsilon (t)$ at infinity in space, for $t > 0$, in particular if $\psi_\textnormal{in}^\varepsilon$ and $\phi_\textnormal{in}^\varepsilon$ do not have a trivial behavior at infinity. Of course, if $\psi_\textnormal{in}^\varepsilon$ and $\phi_\textnormal{in}^\varepsilon$ are analytic themselves, we have the following properties.

\begin{cor} \label{cor:psi_phi_anal_2}
    Assume that $(\psi_\textnormal{in}^\varepsilon, \phi_\textnormal{in}^\varepsilon)$ satisfies Assumption \ref{ass:bound}. Define $T$, $\delta (t)$ and $(\zeta^\varepsilon, v^\varepsilon)$ given by Corollary \ref{cor:cauchy}. If $\psi_\textnormal{in}^\varepsilon \in \mathcal{H}_{\delta_\textnormal{in}}^{\ell+1}$, then
    \begin{equation*}
        \psi^\varepsilon \in \mathcal{C}_T \mathcal{H}_{\delta}^{\ell+\frac{1}{2}} \cap L^\infty_T \mathcal{H}_{\delta}^{\ell+1} \cap L^2_T \mathcal{H}_{\delta}^{\ell+\frac{3}{2}}.
    \end{equation*}
    Furthermore, if we also have $\phi_\textnormal{in}^\varepsilon \in \mathcal{H}_{\delta_\textnormal{in}}^{\ell+2}$, then
    \begin{equation*}
        \phi^\varepsilon \in \mathcal{C}_T \mathcal{H}_{\delta}^{\ell+\frac{3}{2}} \cap L^\infty_T \mathcal{H}_{\delta}^{\ell+2} \cap L^2_T \mathcal{H}_{\delta}^{\ell+\frac{5}{2}}.
    \end{equation*}
\end{cor}

\begin{rem}
    From the second part of Corollary \ref{cor:cauchy}, we know that the behaviour at infinity of $\psi^\varepsilon (t)$ is the same as $\psi_\textnormal{in}^\varepsilon$ for all $t$. Moreover, the behaviour of $\phi^\varepsilon$ is consistent with the effect of scaling for \eqref{eq:lognls} (see Remark \ref{rem:sc_eff}).
    
    In particular, if $\psi_\textnormal{in}^\varepsilon \in \mathcal{H}_{\delta_\textnormal{in}}^\ell$, then $\psi_\textnormal{in}^\varepsilon (x)$ (and then also $\psi^\varepsilon (t,x)$ for all $t \in [0, T]$) goes to 0 when $\abs{x} \rightarrow \infty$, which means that $\abs{u^\varepsilon (t,x)}$ goes to 1 (or any another positive constant if we add a constant to $\psi^\varepsilon$ with Remark \ref{rem:sc_eff}) when $x \rightarrow \infty$. This is therefore linked to the Gross-Pitaevskii problem.
    
    Yet, if $\psi_\textnormal{in}^\varepsilon$ goes to $- \infty$ at infinity, then so does $\psi^\varepsilon (t)$ for any $t \in [0, T]$, which means that we are close to vacuum at infinity at any time for \eqref{eq:lognls} and \eqref{eq:iso_euler}. More generally, if $\psi_\textnormal{in}^\varepsilon$ is bounded by above, then so are $\psi^\varepsilon (t)$ and $\abs{u^\varepsilon (t)}$ for any $t \in [0, T]$. Moreover, in any case and for any compact subset $K \subset \mathbb{R}^d$ and any $k \in \mathbb{N}$, $\psi^\varepsilon (t)$ and $u^\varepsilon (t)$ are $\mathcal{C}^k (K)$ and all the above convergences and properties hold by substituting the analytic spaces in space by $\mathcal{C}^k (K)$ due to the fact $\mathcal{H}^\ell_\delta \subset \mathcal{C}^k (K)$ for any $\delta > 0$ and $\ell \in \mathbb{N}$.
\end{rem}

\subsubsection{Semiclassical limit}

The convergence given in both Theorem \ref{th:main_th} and Corollary \ref{cor:cauchy} suffices to infer the convergence of quadratic observables in some way as soon as the initial data converge. Therefore, we state the following assumptions (we recall that $\tilde{D}_k^\varepsilon$ is defined in \eqref{eq:def_D_tilde}):

\begin{assump} \label{ass:conv1}
    There exists $C > 0$ such that:
    \begin{equation*}
        \tilde{D}_{\ell + \frac{1}{2}}^\varepsilon \leq C \sqrt{\varepsilon}.
    \end{equation*}
\end{assump}

\begin{assump} \label{ass:conv2}
    $\ell > \frac{d+1}{2}$ and there exists $C > 0$ such that:
    \begin{equation*}
        D_{\ell}^\varepsilon \leq C \varepsilon.
    \end{equation*}
\end{assump}

\begin{cor}
    Under Assumption \ref{ass:bound} and Assumption \ref{ass:conv1} or \ref{ass:conv2}, the position and momentum densities converge in the following sense: for any compact subset $K \subset \mathbb{R}^d$, any $k \in \mathbb{N}$ and all $T' < \frac{\delta_\textnormal{in}}{M}$ such that $T' \leq T$, there holds
    \begin{equation*}
        \abs{u^\varepsilon}^2 \underset{\varepsilon \rightarrow 0}{\longrightarrow} e^{\psi^0}, \quad \text{ and } \quad \Im (\varepsilon \overline{u^\varepsilon} \nabla u^\varepsilon) \underset{\varepsilon \rightarrow 0}{\longrightarrow} e^{\psi^0} \, v^0, \quad \text{ in } L^\infty ((0, T'), \mathcal{C}^k (K)).
    \end{equation*}
    Furthermore, if all $\psi^\varepsilon_\textnormal{in} (x)$ are uniformly bounded by above, then all $\psi^\varepsilon (t,x)$ are uniformly bounded by above in $(0, T') \times \mathbb{R}^d$ and there holds for all $k \in \mathbb{N}$
    \begin{equation*}
        \tnorm{\abs{u^\varepsilon}^2 - e^{\psi^0}}_{L_{T'}^\infty H^k} + \tnorm{\Im (\varepsilon \overline{u^\varepsilon} \nabla u^\varepsilon) - e^{\psi^0} \, v^0}_{L_{T'}^\infty H^k} \underset{\varepsilon \rightarrow 0}{\longrightarrow} 0.
    \end{equation*}
\end{cor}

However, regarding convergence of the wave functions, the previous result is not sufficient. Indeed, as fast as $\phi^\varepsilon_\textnormal{in}$ and $\psi^\varepsilon_\textnormal{in}$ may converge as $\varepsilon \rightarrow 0$, Corollary \ref{cor:cauchy} guarantees at most that $\phi^\varepsilon - \phi^0 = O(\varepsilon)$, which only ensures that $e^{\frac{\psi^\varepsilon}{2} + i \frac{\phi^\varepsilon}{\varepsilon}} - e^{\frac{\psi^0}{2} + i \frac{\phi^0}{\varepsilon}} = O(1)$ due to the rapid oscillations. In order to get a better approximation, we have to approximate $\phi^\varepsilon$ up to an error $o(\varepsilon)$ by adding a corrective term. For this purpose, we consider the system obtained by linearizing \eqref{sys:wkb_riemann} around $(\psi^0, \phi^0)$, with $(\psi_{\textnormal{in}, 1}, \phi_{\textnormal{in}, 1})$ real-valued initial data:
\begin{System} \label{sys:wkb_riemann2}
    \partial_t \phi_1 + v^0 \cdot \nabla \phi_1 + \lambda \Re \psi_1 = 0, \qquad \qquad &\phi_1 (0) = \phi_{\textnormal{in}, 1}, \\
    \partial_t \psi_1 + v^0 \cdot \nabla \psi_1 + \nabla \phi_1 \cdot \zeta^0 + \Delta \phi_1 = \frac{i}{2} \Bigl( \divg \zeta^0 + 2 \, \zeta^0 \cdot \zeta^0 \Bigr), \qquad \qquad  &\psi_1 (0) = \psi_{\textnormal{in}, 1}.
\end{System}
In the same way as previously, determining $\phi^\varepsilon_1$ and $\psi^\varepsilon_1$ is equivalent to determining $v^\varepsilon_1 = \nabla \phi^\varepsilon_1$ and $\zeta^\varepsilon_1 = \nabla \psi^\varepsilon_1$ solution to
\begin{System} \label{sys:euler_modified_semicla2}
    \partial_t v_1 + \nabla ( v^0 \cdot v_1 ) + \lambda \, \Re \zeta_1 = 0, \qquad \qquad &v_1 (0) = \nabla \phi_{\textnormal{in}, 1}, \\
    \partial_t \zeta_1 + \nabla ( v^0 \cdot \zeta_1 ) + \nabla ( v_1 \cdot \zeta^0) + \nabla \divg v_1 = \frac{i}{2} \Bigl( \nabla \divg \zeta^0 + 2 \, \nabla ( \zeta^0 \cdot \zeta^0 ) \Bigr), \qquad \qquad  &\zeta_1 (0) = \nabla \psi_{\textnormal{in}, 1},
\end{System}
along with the relations
\begin{gather*}
    \psi_1 (t, x) = \psi_{\textnormal{in}, 1} + \int_0^t \Bigl[ \frac{i}{2} \Bigl( \divg \zeta^0 (\tau, x) + 2 \, \zeta^0 (\tau, x) \cdot \zeta^0 (\tau, x) \Bigr) - v^0 (\tau, x) \cdot \zeta_1 (\tau, x) - v_1 (\tau, x) \cdot \zeta^0 (\tau, x) \Bigr] \diff \tau, \\
    \phi_1 (t,x) = \phi_{\textnormal{in}, 1} - \int_0^t \Bigl( v^0 (\tau, x) \cdot v_1 (\tau, x) + \lambda \Re \psi_1 (\tau, x) \Bigr) \diff \tau.
\end{gather*}

\begin{rem}
    This is also equivalent to linearize \eqref{sys:euler_modified_semicla} around $(\zeta^0, v^0)$.
\end{rem}

Provided $(\zeta^0, v^0)$ satisfying the conclusion of Theorem \ref{th:main_th} and $(\nabla \psi_{\textnormal{in}, 1}, \nabla \phi_{\textnormal{in}, 1}) \in \mathcal{H}_{\delta_\textnormal{in}}^m \times \mathcal{H}_{\delta_\textnormal{in}}^{m+1}$ with $\frac{d-1}{2} < m \leq l-1$, we will see that the solution $(\zeta_1, v_1)$ to \eqref{sys:euler_modified_semicla2} belongs to $(L^\infty_T \mathcal{H}_\delta^{m} \times L^\infty_T \mathcal{H}_\delta^{m+1}) \cap (L^2_T \mathcal{H}_\delta^{m+\frac{1}{2}} \times L^2_T \mathcal{H}_\delta^{m+\frac{3}{2}})$.
Similarly, if $(\psi_{\textnormal{in}, 1}, \phi_{\textnormal{in}, 1}) \in \mathcal{H}_{\delta_\textnormal{in}}^{m+1} \times \mathcal{H}_{\delta_\textnormal{in}}^{m+2}$, we will see that the solution $(\psi_1, \phi_1)$ to \eqref{sys:wkb_riemann2} belongs to $(L^\infty_T \mathcal{H}_\delta^{m+1} \times L^\infty_T \mathcal{H}_\delta^{m+2}) \cap (L^2_T \mathcal{H}_\delta^{m+\frac{3}{2}} \times L^2_T \mathcal{H}_\delta^{m+\frac{5}{2}})$.
The appearance of these correctors, and in particular regarding cases where they are trivial or not, have already been discussed in \cite{Carles_book} in a more classical WKB framework. However, our context is a bit more particular, and we have the following:
\begin{lem} \label{lem:semicla_cond}
    $\phi_1 (t) \equiv 0$ on $[0, T]$ if and only if $(\psi_{\textnormal{in}, 1}, \phi_{\textnormal{in}, 1}) \equiv (0,0)$.
\end{lem}

From the previous discussion, we make the following assumptions:
\begin{assump} \label{ass:semicla1}
    $\ell > \frac{d+3}{2}$ and there exists $(\psi_{\textnormal{in}, 1}, \phi_{\textnormal{in}, 1})$ such that
    \begin{equation*}
        r_{\ell - 2}^\varepsilon \coloneqq \norm{\nabla \psi_\textnormal{in}^\varepsilon - (\nabla \psi_\textnormal{in}^0 + \varepsilon \nabla \psi_{\textnormal{in}, 1})}_{\ell - 2, \delta} + \norm{\nabla \phi_\textnormal{in}^\varepsilon - (\nabla \phi_\textnormal{in}^0 + \varepsilon \nabla \phi_{\textnormal{in}, 1})}_{\ell - 1, \delta} = o(\varepsilon) \qquad \text{as } \varepsilon \rightarrow 0.
    \end{equation*}
\end{assump}
\begin{assump} \label{ass:semicla2}
    $\ell > \frac{d+3}{2}$ and there exists $(\psi_{\textnormal{in}, 1}, \phi_{\textnormal{in}, 1})$ such that
    \begin{equation*}
        \tilde{r}_{\ell - 1}^\varepsilon \coloneqq \norm{\psi_\textnormal{in}^\varepsilon - (\psi_\textnormal{in}^0 + \varepsilon \psi_{\textnormal{in}, 1})}_{\ell - 1, \delta} + \norm{\phi_\textnormal{in}^\varepsilon - (\phi_\textnormal{in}^0 + \varepsilon \phi_{\textnormal{in}, 1})}_{\ell, \delta} = o(\varepsilon) \qquad \text{as } \varepsilon \rightarrow 0.
    \end{equation*}
\end{assump}

Our final result is stated as follows.

\begin{theorem} \label{th:semicla_th}
    Under Assumption \ref{ass:semicla1}, there exists an $\varepsilon$-independent $C > 0$ such that for all $\varepsilon \in [0, 1]$,
    \begin{gather*}
        \tnorm{\zeta^\varepsilon - (\zeta^0 + \varepsilon \zeta_1)}_{\infty, T, \ell - 2, \delta} + \tnorm{\zeta^\varepsilon - (\zeta^0 + \varepsilon \zeta_1)}_{2, T, \ell - \frac{3}{2}, \delta} \leq C (r_{\ell - 2}^\varepsilon + \varepsilon^2), \\
        \tnorm{v^\varepsilon - (v^0 + \varepsilon v_1)}_{\infty, T, \ell - 1, \delta} + \tnorm{v^\varepsilon - (v^0 + \varepsilon v_1)}_{2, T, \ell - \frac{1}{2}, \delta} \leq C (r_{\ell - 2}^\varepsilon + \varepsilon^2).
    \end{gather*}
    Moreover, if Assumption \ref{ass:semicla2} is satisfied, then there also holds for all $\varepsilon \in [0, 1]$,
    \begin{gather*}
        \tnorm{\psi^\varepsilon - (\psi^0 + \varepsilon \psi_1)}_{\infty, T, \ell - 1, \delta} + \tnorm{\psi^\varepsilon - (\psi^0 + \varepsilon \psi_1)}_{2, T, \ell - \frac{1}{2}, \delta} \leq C (\tilde{r}_{\ell - 1}^\varepsilon + \varepsilon^2), \\
        \tnorm{\phi^\varepsilon - (\phi^0 + \varepsilon \phi_1)}_{\infty, T, \ell, \delta} + \tnorm{\phi^\varepsilon - (\phi^0 + \varepsilon \phi_1)}_{2, T, \ell + \frac{1}{2}, \delta} \leq C (\tilde{r}_{\ell - 1}^\varepsilon + \varepsilon^2).
    \end{gather*}
    In particular, for any compact subset $K \subset \mathbb{R}^d$ and $k \in \mathbb{N}$ and $T' < \frac{\delta_\textnormal{in}}{M}$ such that $T' \leq T$, there holds
    \begin{equation*}
        \norm{u^\varepsilon - e^{\frac{\psi^0}{2} + i \phi_1 + i \frac{\phi^0}{\varepsilon}}}_{L^\infty_{T'} \mathcal{C}^k (K)} = O \Bigl( \frac{r_1^\varepsilon}{\varepsilon} + \varepsilon \Bigr) \underset{\varepsilon \rightarrow 0}{\longrightarrow} 0
    \end{equation*}
    and, if $\psi^\varepsilon_\textnormal{in}$ is uniformly bounded by above,
    \begin{equation*}
        \norm{u^\varepsilon - e^{\frac{\psi^0}{2} + i \phi_1 + i \frac{\phi^0}{\varepsilon}}}_{L^\infty_{T'} \mathcal{C}^k_b (\mathbb{R}^d)} = O \Bigl( \frac{r_1^\varepsilon}{\varepsilon} + \varepsilon \Bigr) \underset{\varepsilon \rightarrow 0}{\longrightarrow} 0
    \end{equation*}
\end{theorem}

\subsection{Outline}

In Section \ref{sec:Cauchy}, we first state a toolbox lemma for the computations in analytic spaces, and then address the Cauchy theory in Theorem \ref{th:main_th} in two steps. First, we prove the existence part thanks to a scheme defined in Section \ref{subsec:set_scheme}. Then, we show the uniqueness of this solution through similar estimates as in the existence part. Section \ref{sec:semicla} is devoted to the semiclassical limit, with the proof of the second part of Theorem \ref{th:main_th}. We prove the results about $\psi^\varepsilon$ and $v^\varepsilon$, \textit{i.e.} the second and third parts of Corollary \ref{cor:cauchy}, in Section \ref{sec:cor_proof}. Section \ref{sec:semicla_lim} is devoted to the semiclassical limit of the wave function: we address there Lemma \ref{lem:semicla_cond} and Theorem \ref{th:semicla_th}. Last, we discuss in Section \ref{sec:ass_in_data} the assumptions on the initial data, and in particular the differences from the direct assumption $(\psi^\varepsilon_\textnormal{in}, \nabla \phi^\varepsilon_\textnormal{in}) \in (\mathcal{H}_{\delta_\textnormal{in}}^{\ell + 1})^2$ for instance.

\section*{Acknowledgments}

The author wishes to thank Rémi Carles and Matthieu Hillairet for enlightening and constructive discussions about this work and the writing of this paper.

\section{Cauchy problem} \label{sec:Cauchy}

In this section, we prove Theorem \ref{th:main_th}. Our proof is based on an iterative scheme in a similar way as in \cite{carles-gallo} for example even though it is a little different.

\subsection{Analytic functions} 

We recall that the analytic spaces have been defined in Section \ref{subsubsec:analytic}. We first recall some properties of these spaces (see \cite{Ginibre_Velo__Gevrey}).

\begin{lem} \label{lem:anal_prop}
    Let $\ell, \delta > 0$.
    \begin{enumerate}
    \item For any $\alpha \in \mathbb{N}^d$ and $f \in \mathcal{H}_{\delta}^{\ell + \abs{\alpha}}$,
        \begin{equation*}
            \norm{\partial^\alpha_x f}_{\ell, \delta} \leq \norm{f}_{\ell + \abs{\alpha}, \delta}.
        \end{equation*}
    More precisely, we have:
    \begin{equation} \label{eq:anal_deriv}
        \norm{f}_{\ell+1, \delta}^2 = \norm{f}_{\ell, \delta}^2 + \sum_{\abs{\alpha} = 1} \norm{\partial_x^\alpha f}_{\ell, \delta}^2.
    \end{equation}
    
    \item For any $s \in \mathbb{R}$, $f \in \mathcal{H}_{\delta}^{\ell+s}$ and $g \in \mathcal{H}_{\delta}^{\ell-s}$,
        \begin{equation*}
            \langle f, g \rangle_{\mathcal{H}_{\delta}^\ell} \leq \norm{f}_{\ell+s, \delta} \norm{g}_{\ell-s, \delta}.
        \end{equation*}

    \item \label{part3} For any $m > \frac{d}{2}$, there exists $K^{\ell, m} > 0$ (if $\ell = m$, we will simply denote it by $K^\ell$) which does not depend on $\delta > 0$ such that for any $f,g \in \mathcal{H}_\delta^{\max (m, \ell)}$,
        \begin{equation*}
            \norm{f \cdot g}_{\ell, \delta} \leq \frac{1}{2} K^{\ell, m} \Bigl( \norm{f}_{m, \delta} \norm{g}_{\ell, \delta} + \norm{f}_{\ell, \delta} \norm{g}_{m, \delta} \Bigr).
        \end{equation*}

    \item For any $f \in \mathcal{H}_{\delta}^{\ell}$, if $f$ is scalar, then
        \begin{equation*}
            \Re \langle f, i \Delta f \rangle_{\ell, \delta} = 0;
        \end{equation*}
        if $f$ is $\mathbb{R}^d$-valued, then
        \begin{equation*}
            \Re \langle f, i \nabla \divg f \rangle_{\ell, \delta} = 0;
        \end{equation*}
    
    \item If $\ell > \frac{d}{2}$, we have a constant $C > 0$ such that for all $\delta > 0$ and all $f \in \mathcal{H}_\delta^\ell$,
        \begin{equation*}
            \norm{f}_{L^\infty} \leq C \norm{f}_{H^\ell} \leq C \norm{f}_{\ell,\delta}.
        \end{equation*}
    \end{enumerate}
\end{lem}

Moreover, as already said in Section \ref{subsubsec:analytic}, we take time-depending $\delta (t)$ for the analytic regularity $\mathcal{H}_\delta^\ell$. $\norm{f}_{\ell,\delta}^2$ for time-depending $f$ and $\delta$ can be estimated thanks to the following result.

\begin{lem} \label{lem:prop_anal_norm_diff}
    For a $\mathcal{C}^1$ time-dependent $\delta$, we have:
    \begin{equation*}
        \frac{\diff}{\diff t} \norm{f}_{\ell,\delta}^2 = 2 \dot{\delta} \norm{f}_{\ell + \frac{1}{2}, \delta}^2 + 2 \Re \langle f, \partial_t f \rangle_{\ell, \delta}.
    \end{equation*}
\end{lem}

The following lemma, based on the previous properties, is a toolbox for all the forthcoming analysis and estimates. For a partial proof, we refer to \cite{carles-gallo}, most cases not treated in there can be treated in a similar way thanks to Lemma \ref{lem:anal_prop}.

\begin{lem} \label{lem:toolbox}
    Let $m > \frac{d}{2} - 1$, $M > 0$, $\delta (t) = \delta_\textnormal{in} - M t$ and $T \leq \frac{\delta_\textnormal{in}}{M}$. Let $(f, g) \in \mathcal{C} ([0, T], \mathcal{H}_\delta^{m+\frac{1}{2}} \times \mathcal{H}_\delta^{m-\frac{1}{2}} )$, and denote $\operatorname{D} \coloneqq \Delta$ if they are $\mathbb{C}$-valued or $\operatorname{D} \coloneqq \nabla \divg$ if they are $\mathbb{C}^d$-valued. Let $(F, G) \in L^2 ((0, T), \mathcal{H}_\delta^{m+\frac{1}{2}} \times \mathcal{H}_\delta^{m-\frac{1}{2}} )$, $\tilde{g}_1 \in L^2_T \mathcal{H}_\delta^{m + 1}$, $\tilde{g}_2 \in L^\infty_T \mathcal{H}_\delta^{m + 1}$ and $\theta_1, \theta_2, \theta_3 \in \mathbb{R}$, and assume that $g \in L^2_T \mathcal{H}^{m+1}_\delta$ if $\theta_2 \neq 0$ and
    \begin{equation*}
        \partial_t f = F, \qquad \qquad f(0) \in \mathcal{H}_{\delta_\textnormal{in}}^{m+1},
    \end{equation*}
    \begin{equation*}
        \partial_t g = G + i \theta_1 \operatorname{D} g + i \theta_2 \operatorname{D} \tilde{g}_1 + i \theta_3 \nabla (\tilde{g}_2 \cdot \tilde{g}_2), \qquad g(0) \in \mathcal{H}_{\delta_\textnormal{in}}^{m}.
    \end{equation*}
    Then, $t \mapsto \norm{f (t)}_{m+1, \delta (t)}^2$ and $t \mapsto \norm{g (t)}_{m, \delta (t)}^2$ are continuous and for all $t \in [0, T]$,
    \begin{equation*}
        \mathcal{E}_{M, m + 1, \delta} (f) (t) \leq \norm{f(0)}_{m+1, \delta}^2 + 2 \tnorm{f}_{2, t, {m+\frac{3}{2}}, \delta} \tnorm{F}_{2, t, {m+\frac{1}{2}}, \delta},
    \end{equation*}
    \begin{multline*}
        \mathcal{E}_{M, m, \delta} (g) (t) \leq \norm{g(0)}_{m, \delta}^2 + 2 \tnorm{g}_{2, t, {m+\frac{1}{2}}, \delta} \tnorm{G}_{2, t, {m-\frac{1}{2}}, \delta} \\ + 2 \abs{\theta_2} \tnorm{g}_{2, t, {m+1}, \delta} \tnorm{\tilde{g}_1}_{2, t, {m+1}, \delta} + 2 T \abs{\theta_3} \tnorm{g}_{\infty, t, m, \delta} \tnorm{\tilde{g}_2}_{\infty, t, {m+1}, \delta}^2.
    \end{multline*}
    In the case $\theta_2 = 0$, the term $2 \abs{\theta_2} \norm{g}_{2, {m+1}, \delta} \norm{\tilde{g}_1}_{2, {m+1}, \delta}$ should be understood to be zero in any case.
    Moreover, there holds for all $t \in [0, T]$
    \begin{itemize}
    \item If $F = F_1 \cdot F_2$ with $F_1 \in L^\infty_T \mathcal{H}_\delta^{m+\frac{1}{2}}$ and $F_2 \in L^2_T \mathcal{H}_\delta^{m+\frac{3}{2}}$ and $m > \frac{d-1}{2}$, then
    \begin{equation} \label{eq:tool0}
        \tnorm{F}_{2, t, {m+\frac{1}{2}}, \delta} \leq K^{m+\frac{1}{2}} \tnorm{F_1}_{\infty, t, m+\frac{1}{2}, \delta} \tnorm{F_2}_{2, t, {m+\frac{3}{2}}, \delta},
    \end{equation}
    where $K^{m+\frac{1}{2}}$ is defined in Lemma \ref{lem:anal_prop} (part \ref{part3}).

    \item If $F = (F_1 \cdot \nabla) F_2$ with $F_1 \in L^\infty_T \mathcal{H}_\delta^{m+\frac{1}{2}}$ and $F_2 \in L^2_T \mathcal{H}_\delta^{m+\frac{3}{2}}$ and $m > \frac{d-1}{2}$, then
    \begin{equation} \label{eq:tool1}
        \tnorm{F}_{2, t, {m+\frac{1}{2}}, \delta} \leq K^{m+\frac{1}{2}} \tnorm{F_1}_{\infty, t, m+\frac{1}{2}, \delta} \tnorm{F_2}_{2, t, {m+\frac{3}{2}}, \delta}.
    \end{equation}

    \item If $F = (F_1 \cdot \nabla) F_2$ with $F_1 \in L^2_T \mathcal{H}_\delta^{m+\frac{1}{2}}$ and $F_2 \in L^\infty_T \mathcal{H}_\delta^{m+\frac{3}{2}}$ and $m > \frac{d-1}{2}$, then
    \begin{equation}
        \tnorm{F}_{2, t, {m+\frac{1}{2}}, \delta} \leq K^{m+\frac{1}{2}} \tnorm{F_1}_{2, t, m+\frac{1}{2}, \delta} \tnorm{F_2}_{\infty, t, {m+\frac{3}{2}}, \delta}.
    \end{equation}

    \item If $F = \theta_4 \Re F_1$ with $F_1 \in L^2_T \mathcal{H}_\delta^{m + \frac{1}{2}}$ and $\theta_4 \in \mathbb{R}$, then
    \begin{equation} \label{eq:tool3}
        \tnorm{F}_{2, t, {m+\frac{1}{2}}, \delta} \leq 2 \abs{\theta_4} \tnorm{F_1}_{\infty, t, m+\frac{1}{2}, \delta}.
    \end{equation}

    \item If $F = \nabla (F_1 \cdot F_2)$ with $F_1 \in L^\infty_T \mathcal{H}_\delta^{m+\frac{3}{2}}$ and $F_2 \in L^2_T \mathcal{H}_\delta^{m+\frac{3}{2}}$, then
    \begin{equation} \label{eq:tool4}
        \tnorm{F}_{2, t, {m+\frac{1}{2}}, \delta} \leq K^{m + \frac{3}{2}} \tnorm{F_1}_{\infty, t, m+\frac{3}{2}, \delta} \tnorm{F_2}_{2, t, {m+\frac{3}{2}}, \delta}.
    \end{equation}

    \item If $G = \nabla (G_1 \cdot G_2)$ with $G_1 \in L^\infty_T \mathcal{H}_\delta^{m+\frac{1}{2}}$ and $G_2 \in L^2_T \mathcal{H}_\delta^{m+\frac{1}{2}}$ and $m > \frac{d-1}{2}$, then
    \begin{equation} \label{eq:tool5}
        \tnorm{G}_{2, t, {m-\frac{1}{2}}, \delta} \leq K^{m + \frac{1}{2}} \tnorm{G_1}_{\infty, m+\frac{1}{2}, \delta} \tnorm{G_2}_{2, t, {m+\frac{1}{2}}, \delta}.
    \end{equation}

    \item If $G = \theta_5 \operatorname{D} G_1$ with $G_1 \in L^2_T \mathcal{H}_\delta^{m+\frac{3}{2}}$ and $\theta_5 \in \mathbb{C}$, then
    \begin{equation} \label{eq:tool6}
        \tnorm{G}_{2, t, {m-\frac{1}{2}}, \delta} \leq \abs{\theta_5} \tnorm{G_1}_{2, t, {m+\frac{3}{2}}, \delta}
    \end{equation}

    \item If $G = \theta_6 \nabla (G_1 \cdot G_1)$ with $G_1 \in L^\infty_T \mathcal{H}_\delta^{m} \cap L^2_T \mathcal{H}_\delta^{m+\frac{1}{2}}$ and $\theta_6 \in \mathbb{C}$ and $m > \frac{d}{2}$, then
    \begin{equation} \label{eq:tool7}
        \tnorm{G}_{2, t, {m-\frac{1}{2}}, \delta} \leq \abs{\theta_6} K^{m + \frac{1}{2}, m} \tnorm{G_1}_{\infty, t, m, \delta} \tnorm{G_1}_{2, t, {m+\frac{1}{2}}, \delta}.
    \end{equation}

    \item If $G = G_1 \cdot G_2$ with $G_1 \in L^\infty_T \mathcal{H}_\delta^{m+1}$ and $G_2 \in L^2_T \mathcal{H}_\delta^{m+1}$, then
    \begin{equation} \label{eq:tool8}
        \tnorm{G}_{2, t, {m-\frac{1}{2}}, \delta} \leq K^{m + 1} \tnorm{G_1}_{\infty, t, m+1, \delta} \tnorm{G_2}_{2, t, {m+1}, \delta}.
    \end{equation}
    or, if $m > \frac{d-1}{2}$,
    \begin{equation} \label{eq:tool9}
        \tnorm{G}_{2, t, {m-\frac{1}{2}}, \delta} \leq K^{m + \frac{1}{2}} \tnorm{G_1}_{\infty, t, m+\frac{1}{2}, \delta} \tnorm{G_2}_{2, t, {m+\frac{1}{2}}, \delta}.
    \end{equation}

    \item If $G = \theta_7 \divg G_1$ with $G_1 \in L^2_T \mathcal{H}_\delta^{m+\frac{1}{2}}$ and $\theta_7 \in \mathbb{C}$, then
    \begin{equation} \label{eq:tool10}
        \tnorm{G}_{2, t, {m-\frac{1}{2}}, \delta} \leq \abs{\theta_7} \tnorm{G_1}_{2, t, {m+\frac{1}{2}}, \delta}
    \end{equation}
    \end{itemize}
    %
\end{lem}

\subsection{Setting of the scheme} \label{subsec:set_scheme}

Let $\varepsilon \in [0, 1]$.
Set $\zeta^\varepsilon_0 (t) \coloneqq \nabla \psi_\textnormal{in}^\varepsilon$ and $v^\varepsilon_0 (t) \coloneqq \nabla \phi_\textnormal{in}^\varepsilon$ for all $t \geq 0$.
Then, for all $k \in \mathbb{N}$, define $\zeta^\varepsilon_{k+1}$ and $v^\varepsilon_{k+1}$ by induction as the solution to
\begin{System} \label{sys:scheme_zeta_v}
    \partial_t v^\varepsilon_{k+1} + (v^\varepsilon_k \cdot \nabla) v^\varepsilon_{k+1} + \lambda \, \Re \zeta^\varepsilon_k = 0,\qquad \qquad v^\varepsilon_{k+1} (0) = \nabla \phi_\textnormal{in}^\varepsilon, \\
    \partial_t \zeta^\varepsilon_{k+1} + \nabla \Bigl( v^\varepsilon_k \cdot \zeta^\varepsilon_k \Bigr) + \nabla \divg v^\varepsilon_{k+1} = i \frac{\varepsilon}{2} \Bigl( \nabla \divg \zeta^\varepsilon_{k+1} + 2 \nabla ( \zeta^\varepsilon_k \cdot \zeta^\varepsilon_k ) \Bigr), \qquad \qquad  \zeta^\varepsilon_{k+1} (0) = \nabla \psi_\textnormal{in}^\varepsilon.
\end{System}
The first equation is an explicit transport equation with source term and does not depend on $\zeta^\varepsilon_{k+1}$ so that $v^\varepsilon_{k+1}$ can be defined first independently. For our case, we will show that those terms are smooth (and even analytic).
Then, the second equation can be solved thanks to the Schrödinger semigroup:
\begin{equation} \label{eq:def_zeta_k_scheme}
    \zeta^\varepsilon_{k+1} (t) = \nabla e^{i \frac{\varepsilon}{2} t \Delta} \psi_\textnormal{in}^\varepsilon - \nabla \int_0^t e^{i \frac{\varepsilon}{2} (t - \tau) \Delta} \Bigl( v^\varepsilon_k (\tau) \cdot \zeta^\varepsilon_k (\tau) + \frac{1}{2} \divg v^\varepsilon_{k+1} (\tau) - i \frac{\varepsilon}{2} \zeta^\varepsilon_k (\tau) \cdot \zeta^\varepsilon_k (\tau) \Bigr) \diff \tau.
\end{equation}
It is easy to see that $\zeta^\varepsilon_{k+1}$ defined by \eqref{eq:def_zeta_k_scheme} satisfies \eqref{sys:scheme_zeta_v}. Indeed, define
\begin{equation*}
    \psi^\varepsilon_{k+1} (t) = e^{i \frac{\varepsilon}{2} t \Delta} \psi_\textnormal{in}^\varepsilon - \int_0^t e^{i \frac{\varepsilon}{2} (t - \tau) \Delta} \Bigl( v^\varepsilon_k (\tau) \cdot \zeta^\varepsilon_k (\tau) + \frac{1}{2} \divg v^\varepsilon_{k+1} (\tau) - i \frac{\varepsilon}{2} \zeta^\varepsilon_k (\tau) \cdot \zeta^\varepsilon_k (\tau) \Bigr) \diff \tau,
\end{equation*}
then it easy to check that $\zeta^\varepsilon_{k+1} = \nabla \psi^\varepsilon_{k+1}$ and
\begin{equation*}
    \partial_t \psi^\varepsilon_{k+1} - i \frac{\varepsilon}{2} \Delta \psi^\varepsilon_{k+1} = - \Bigl( v^\varepsilon_k \cdot \zeta^\varepsilon_k + \frac{1}{2} \divg v^\varepsilon_{k+1} - i \frac{\varepsilon}{2} \zeta^\varepsilon_k \cdot \zeta^\varepsilon_k \Bigr).
\end{equation*}

\subsection{Well-posedness of the scheme} \label{subsec:wellpos_scheme}

Fix now $\psi_\textnormal{in}^\varepsilon, \phi_\textnormal{in}^\varepsilon$ satisfying Assumption \ref{ass:bound}.
With this assumption, our scheme is well-posed (at least locally in time).

\begin{lem} \label{lem:scheme_unif_bound}
    There exists $M > 0$ and $T \in (0, \frac{\delta_\textnormal{in}}{M}]$ such that, for $\delta (t) \coloneqq \delta_\textnormal{in} - M t$, $(v^\varepsilon_k, \zeta^\varepsilon_k)$ is well defined and uniformly bounded in both $k \in \mathbb{N}$ and $\varepsilon \in [0, 1]$ in $\mathcal{C} ([0, T], \mathcal{H}_{\delta}^{\ell+1} \times \mathcal{H}_{\delta}^{\ell}) \cap L^2 ((0, T), \mathcal{H}_{\delta}^{\ell+\frac{3}{2}} \times \mathcal{H}_{\delta}^{\ell + \frac{1}{2}})$.
\end{lem}

\begin{proof}
We show this result by induction. The fact that $\zeta_0^\varepsilon$ and $v_0^\varepsilon$ are well defined is obviously true.
Since $\partial_t v^\varepsilon_0 = \partial_t \zeta^\varepsilon_0 = 0$, Lemma \ref{lem:toolbox} gives for all $t \geq 0$ with $\delta (t) = \delta_\textnormal{in} - M t$:
\begin{equation} \label{eq:est_scheme_0}
    \mathcal{E}_{M, \ell, \delta} (\zeta^\varepsilon_0) (t) + \mathcal{E}_{M, \ell + 1, \delta} (v^\varepsilon_0) (t) \leq \norm{\zeta^\varepsilon_\textnormal{in}}_{\ell,\delta_\textnormal{in}}^2 + \norm{v^\varepsilon_\textnormal{in}}_{\ell+1,\delta_\textnormal{in}}^2 \leq \omega_\textnormal{in}.
\end{equation}
Therefore, we have $(\zeta_0^\varepsilon, v_0^\varepsilon) \in L^\infty ((0, T), \mathcal{H}_{\delta}^{\ell+1} \times \mathcal{H}_{\delta}^{\ell}) \cap L^2 ((0, T), \mathcal{H}_{\delta}^{\ell+\frac{3}{2}} \times \mathcal{H}_{\delta}^{\ell + \frac{1}{2}})$ as long as we take $M > 0$.
Now, assume that it is true for some $k \geq 0$. With this property, $v^\varepsilon_{k+1}$ is solution of a transport equation with explicit smooth terms and is therefore well defined (thanks to characteristics). Then, $\zeta^\varepsilon_{k+1}$ is also well defined thanks to \eqref{eq:def_zeta_k_scheme} along with the property of the Schrödinger semigroup in analytic spaces.
Then, we use Lemma \ref{lem:toolbox} with $f = v^\varepsilon_{k+1}$, $g = \zeta^\varepsilon_{k+1}$, $m = \ell$, $\theta_1 = \frac{\varepsilon}{2}$, \eqref{eq:tool1}, \eqref{eq:tool3} with $\theta_4 = \lambda$ and \eqref{eq:tool5}-\eqref{eq:tool7} with $\theta_5 = 1$ and $\theta_6 = i \frac{\varepsilon}{2}$.
For that, set
\begin{gather*}
    \omega^\varepsilon_k (t) \coloneqq \norm{\zeta^\varepsilon_k}_{\infty, t, \ell,\delta}^2 + \norm{v^\varepsilon_k}_{\infty, t, \ell+1,\delta}^2, \\
    \eta^\varepsilon_k (t) \coloneqq \norm{\zeta^\varepsilon_k}_{2, t, \ell+\frac{1}{2},\delta}^2 + \norm{v^\varepsilon_k}_{2, t, \ell+\frac{3}{2},\delta}^2.
\end{gather*}
We also use the following computations:
\begin{equation*}
    2 \, \tnorm{v^\varepsilon_{k+1}}_{2, t, \ell + \frac{3}{2}, \delta} \tnorm{\zeta^\varepsilon_k}_{2, t, \ell + \frac{1}{2}, \delta} \leq \tnorm{v^\varepsilon_{k+1}}_{2, t, \ell + \frac{3}{2}, \delta}^2 + \tnorm{\zeta^\varepsilon_k}_{2, t, \ell + \frac{1}{2}, \delta}^2,
\end{equation*}
\begin{equation*}
    \tnorm{\zeta^\varepsilon_{k+1}}_{2, t, \ell + \frac{1}{2}, \delta} \tnorm{v^\varepsilon_k}_{\infty, t, \ell + \frac{1}{2}, \delta} \tnorm{\zeta^\varepsilon_k}_{2, t, \ell + \frac{1}{2}, \delta} \leq \frac{1}{2} \tnorm{v^\varepsilon_k}_{\infty, t, \ell + 1, \delta}^2 \tnorm{\zeta^\varepsilon_{k+1}}_{2, t, \ell + \frac{1}{2}, \delta}^2 + \frac{1}{2} \tnorm{\zeta^\varepsilon_k}_{2, t, \ell + \frac{1}{2}, \delta}^2,
\end{equation*}
\begin{equation*}
    \tnorm{\zeta^\varepsilon_{k+1}}_{2, t, \ell + \frac{1}{2}, \delta} \tnorm{v^\varepsilon_{k+1}}_{2, t, \ell + \frac{3}{2}, \delta} \leq \frac{1}{2} \eta^\varepsilon_{k+1},
\end{equation*}
\begin{equation*}
    \tnorm{\zeta^\varepsilon_{k+1}}_{2, t, \ell + \frac{1}{2}, \delta} \tnorm{\zeta^\varepsilon_k}_{\infty, t, \ell, \delta} \tnorm{\zeta^\varepsilon_k}_{2, t, \ell + \frac{1}{2}, \delta} \leq \frac{1}{2} \tnorm{\zeta^\varepsilon_k}_{\infty, t, \ell, \delta}^2 \tnorm{\zeta^\varepsilon_{k+1}}_{2, t, \ell + \frac{1}{2}, \delta}^2 + \frac{1}{2} \tnorm{\zeta^\varepsilon_k}_{2, t, \ell + \frac{1}{2}, \delta}^2.
\end{equation*}
Therefore, for all $\varepsilon \leq 1$ and $t \geq 0$ such that $\delta (t) \geq 0$ and using Assumption \ref{ass:bound},
\begin{multline} \label{eq:est_scheme_k}
    \mathcal{E}_{M, \ell, \delta} (\zeta^\varepsilon_{k+1}) (t) + \mathcal{E}_{M, \ell + 1, \delta} (v^\varepsilon_{k+1}) (t) \leq \omega_\textnormal{in} + (2 C_{\ell} (\sqrt{\omega^\varepsilon_k (t)} + \omega^\varepsilon_k (t)) + 2 \abs{\lambda} + 1 ) \eta^\varepsilon_{k+1} (t) \\ + 2 (\abs{\lambda} + C_\ell) \, \tnorm{\zeta^\varepsilon_k}_{2, t, \ell + \frac{1}{2}, \delta}^2,
\end{multline}
for some $C_\ell > 0$ depending only on $\ell$.
Set
\begin{gather*}
    M_1 \coloneqq C_{\ell} (\sqrt{2 \omega_\textnormal{in}} + 2 \omega_\textnormal{in}) + \abs{\lambda} + \frac{1}{2}, \\
    M_2 \coloneqq \abs{\lambda} + C_\ell, \\
    M \coloneqq M_1 + 2 M_2.
\end{gather*}
Moreover, take $\delta (t) = \delta_\textnormal{in} - M t$ and set $T = M^{-1} \delta_\textnormal{in}$ so that $\delta (t) \geq 0$ for all $t \in [0, T]$.
From these estimates and definitions, we can prove that the scheme is uniformly bounded thanks to the following Lemma.
\begin{lem} \label{lem:est_scheme}
    For all $\varepsilon \in [0,1]$, $t \in [0, T]$ and $k \in \mathbb{N}$, there holds
    \begin{equation*}
        \mathcal{E}_{2 M_2, \ell, \delta} (\zeta^\varepsilon_{k}) (t) + \mathcal{E}_{2 M_2, \ell + 1, \delta} (v^\varepsilon_{k}) (t) \leq 2 \omega_\textnormal{in}.
    \end{equation*}
\end{lem}
The proof is therefore complete.
\end{proof}

\begin{proof}[Proof of Lemma \ref{lem:est_scheme}]
    We prove this lemma by induction on $k$.
    The estimate for $k = 0$ follows from \eqref{eq:est_scheme_0} and the facts that
    \begin{equation*}
        - M \leq - 2 M_2.
    \end{equation*}
    Now, for $k \in \mathbb{N}$, assuming that the estimate holds at rank $k$, we have in particular the fact that $\omega^\varepsilon_k (t) \leq 2 \omega_\textnormal{in}$ for all $t \in [0, T]$, so that \eqref{eq:est_scheme_k} becomes
    \begin{equation*}
        \mathcal{E}_{2 M_2, \ell, \delta} (\zeta^\varepsilon_{k+1}) (t) + \mathcal{E}_{2 M_2, \ell + 1, \delta} (v^\varepsilon_{k+1}) (t) \leq \omega_\textnormal{in} + 2 M_2 \eta^\varepsilon_k (t).
    \end{equation*}
    Using again the property at rank $k$, we have for all $t \in [0, T]$
    \begin{equation*}
        2 M_2 \eta^\varepsilon_k (t) \leq \omega_\textnormal{in},
    \end{equation*}
    and thus the property at rank $k+1$ is proved.
\end{proof}

\subsection{Convergence of the scheme} \label{subsec:conv_scheme}

We proved that the scheme is well defined. We now need to show that this scheme converges as $k \rightarrow \infty$ in order to get a solution to \eqref{sys:euler_modified_semicla} from this limit.

\begin{lem} \label{lem:cauchy_seq_scheme}
    Up to taking a larger $M > 0$ and a smaller $T > 0$, for any $\varepsilon \in [0, 1]$, $(\zeta^\varepsilon_k, v^\varepsilon_k)_k$ is a Cauchy sequence in \\ $\mathcal{C} ([0, T],  \mathcal{H}_{\delta}^{\ell-\frac{1}{2}} \times \mathcal{H}_{\delta}^{\ell+\frac{1}{2}}) \cap L^2 ((0, T), \mathcal{H}_{\delta}^{\ell} \times \mathcal{H}_{\delta}^{\ell+1})$.
\end{lem}

\begin{proof}
We know that $(\zeta^\varepsilon_k, v^\varepsilon_k)$ is uniformly bounded in $L^\infty ((0, T), \mathcal{H}_{\delta}^\ell \times \mathcal{H}_{\delta}^{\ell+1}) \cap L^2 ((0, T), \mathcal{H}_{\delta}^{\ell+\frac{1}{2}} \times \mathcal{H}_{\delta}^{\ell+\frac{1}{2}})$ with Lemma \ref{lem:scheme_unif_bound}.
Set $Z^\varepsilon_{k+1} \coloneqq \zeta^\varepsilon_{k+1} - \zeta^\varepsilon_k$ and $V^\varepsilon_{k+1} \coloneqq v^\varepsilon_{k+1} - v^\varepsilon_k$ for $k \in \mathbb{N}$.
Then, we obtain for $k \geq 1$:
\begin{System} \notag
    \partial_t V^\varepsilon_{k+1} + (v^\varepsilon_{k} \cdot \nabla) V^\varepsilon_{k+1} + (V^\varepsilon_k \cdot \nabla) v^\varepsilon_{k} + \lambda \, \Re Z^\varepsilon_k = 0, \\
    \partial_t Z^\varepsilon_{k+1} + \nabla \Bigl( V^\varepsilon_k \cdot \zeta^\varepsilon_k \Bigr) + \nabla \Bigl( v^\varepsilon_{k-1} \cdot Z^\varepsilon_k \Bigr) + \nabla \divg V^\varepsilon_{k+1} = i \frac{\varepsilon}{2} \Bigl( \nabla \divg Z^\varepsilon_{k+1} + 2 \, \nabla ( Z^\varepsilon_k \cdot \zeta^\varepsilon_k ) + 2 \, \nabla ( \zeta^\varepsilon_{k-1} \cdot Z^\varepsilon_k ) \Bigr),
\end{System}
with zero initial data. 
Set
\begin{gather*}
    N^\varepsilon_k (t) \coloneqq \tnorm{Z^\varepsilon_{k}}_{2, t, \ell, \delta}^2 + \tnorm{V^\varepsilon_{k}}_{2, t, \ell + 1, \delta }^2.
\end{gather*}
From the previous system and Lemma \ref{lem:toolbox} with $m = \ell - \frac{1}{2}$, \eqref{eq:tool1}-\eqref{eq:tool3}, \eqref{eq:tool6} and three times \eqref{eq:tool5} in a similar way as previously, along with Lemma \ref{lem:est_scheme} and the following computations:
\begin{equation*}
    \tnorm{V^\varepsilon_{k+1}}_{2, t, \ell+1, \delta} \tnorm{V^\varepsilon_k}_{2, t, \ell, \delta} \tnorm{v^\varepsilon_{k}}_{\infty, t, \ell + 1, \delta} \leq \omega_\textnormal{in} \tnorm{V^\varepsilon_{k+1}}_{2, t, \ell+1, \delta}^2 + \frac{1}{2} \tnorm{V^\varepsilon_k}_{2, t, \ell, \delta}^2,
\end{equation*}
\begin{equation*}
    \tnorm{Z^\varepsilon_{k+1}}_{2, t, \ell, \delta} \tnorm{V^\varepsilon_k}_{2, t, \ell, \delta} \tnorm{\zeta^\varepsilon_k}_{\infty, t, \ell, \delta} \leq \omega_\textnormal{in} \tnorm{Z^\varepsilon_{k+1}}_{2, t, \ell, \delta}^2 + \frac{1}{2} \tnorm{V^\varepsilon_k}_{2, t, \ell, \delta}^2,
\end{equation*}
\begin{equation*}
    \tnorm{Z^\varepsilon_{k+1}}_{2, t, \ell, \delta} \tnorm{Z^\varepsilon_k}_{2, t, \ell, \delta} \sqrt{2 \omega_\textnormal{in}} \leq \omega_\textnormal{in} \tnorm{Z^\varepsilon_{k+1}}_{2, t, \ell, \delta}^2 + \frac{1}{2} \tnorm{Z^\varepsilon_k}_{2, t, \ell, \delta}^2,
\end{equation*}
we get:
\begin{multline*}
    \mathcal{E}_{M, \ell - \frac{1}{2}, \delta} (Z^\varepsilon_{k+1}) (t) + \mathcal{E}_{M, \ell + \frac{1}{2}, \delta} (V^\varepsilon_{k+1}) (t) \leq ( 2 K^\ell \sqrt{2 \omega_\textnormal{in}} + 2 K^\ell \omega_\textnormal{in} (2 + \varepsilon) + 1 + 2 \abs{\lambda} )
        N^\varepsilon_{k+1} (t) \\
            + 2 K^\ell \norm{V^\varepsilon_k}_{2, \ell, \delta}^2 + (2 \abs{\lambda} + K^\ell (1 + 2 \varepsilon)) \tnorm{Z^\varepsilon_k}_{2, t, \ell, \delta}^2.
\end{multline*}
Therefore, for $\varepsilon \leq 1$ and up to taking a slightly larger $C_\ell$ in $M_1$, we get
\begin{equation} \label{eq:est_conv_scheme_1}
    \mathcal{E}_{2 M_2, \ell - \frac{1}{2}, \delta} (Z^\varepsilon_{k+1}) (t) + \mathcal{E}_{2 M_2, \ell + \frac{1}{2}, \delta} (V^\varepsilon_{k+1}) (t) \leq 2 M_2 N^\varepsilon_{k} (t).
\end{equation}
From this estimate, we can prove a uniform estimate.

\begin{lem} \label{lem:unif_sec_est_scheme}
    %
    There holds for all $k \in \mathbb{N}$ and $t \in [0, T]$
    \begin{equation*}
        \mathcal{I}^\varepsilon_{k+1} (t) \coloneqq \mathcal{E}_{2 M_2, \ell - \frac{1}{2}, \delta} (Z^\varepsilon_{k+1}) (t) + \mathcal{E}_{2 M_2, \ell + \frac{1}{2}, \delta} (V^\varepsilon_{k+1}) (t) \leq \omega_\textnormal{in} \, 2^{-k+2}.
    \end{equation*}
\end{lem}

We first finish the proof of Lemma \ref{lem:cauchy_seq_scheme} before proving this Lemma. For any $j > k \geq 1$, there holds with Lemma \ref{lem:unif_sec_est_scheme}:
    \begin{align*}
        \tnorm{\zeta^\varepsilon_j - \zeta^\varepsilon_k}_{\infty, T, \ell - \frac{1}{2}, \delta} &= \tnorm{\sum_{m=k}^{j-1} Z^\varepsilon_{m+1}}_{\infty, T, \ell - \frac{1}{2}, \delta} \\
                &\leq \sum_{m=k}^{j-1} \tnorm{Z^\varepsilon_{m+1} (t)}_{\infty, T, \ell - \frac{1}{2}, \delta} \\
                &\leq \sum_{m=k}^{j-1} \sqrt{\omega_\textnormal{in}} (\sqrt{2})^{-m+2} \\
                &\leq C \sqrt{2}^{-k}.
    \end{align*}
    Therefore $(\zeta^\varepsilon_k)_k$ is a Cauchy sequence in $\mathcal{C} ([0, T],  \mathcal{H}_{\delta}^{\ell-\frac{1}{2}})$ since it is continuous with values in the same space. The same kind of estimate can be proved for $v^\varepsilon_k$ and for the $L^2 ((0, T), \mathcal{H}_{\delta}^{\ell} \times \mathcal{H}_{\delta}^{\ell+1})$ norm in the same way.
\end{proof}

\begin{proof}[Proof of Lemma \ref{lem:unif_sec_est_scheme}]
    \eqref{eq:est_conv_scheme_1} leads to
    \begin{equation*}
        \mathcal{I}_{k+1}^\varepsilon \leq \frac{1}{2} \mathcal{I}_k^\varepsilon.
    \end{equation*}
    Thus, we get
    \begin{equation*}
        \mathcal{I}_{k+1}^\varepsilon \leq 2^{-k} \mathcal{I}_1^\varepsilon.
    \end{equation*}
    Moreover,
    \begin{equation*}
        \tnorm{Z_1^\varepsilon}_{\infty, T, \ell - \frac{1}{2}, \delta} + \tnorm{V_1^\varepsilon}_{\infty, T, \ell + \frac{1}{2}, \delta} \leq \omega^\varepsilon_1 (t) + \omega^\varepsilon_0  (t),
    \end{equation*}
    and
    \begin{equation*}
        2 M_2 N^\varepsilon_1 (t) \leq 2 M_2 \int_0^t \Bigl( \norm{\zeta^\varepsilon_{1} (\tau)}_{\ell, \delta}^2 + \norm{v^\varepsilon_{1} (\tau)}_{\ell + 1, \delta}^2 \Bigr) \diff \tau + 2 M_2 \int_0^t \Bigl( \norm{\zeta^\varepsilon_{0} (\tau)}_{\ell, \delta}^2 + \norm{v^\varepsilon_{0} (\tau)}_{\ell + 1, \delta}^2 \Bigr) \diff \tau.
    \end{equation*}
    Therefore, from Lemma \ref{lem:est_scheme}, we can deduce that
    \begin{equation*}
        \mathcal{I}_1^\varepsilon \leq 4 \omega_\textnormal{in},
    \end{equation*}
    and then the conclusion.
\end{proof}

From Lemma \ref{lem:cauchy_seq_scheme} along with Lemma \ref{lem:est_scheme}, we can complete the proof of the existence of a solution to \eqref{sys:euler_modified_semicla} and the uniform estimates in $\varepsilon$.

\begin{cor}
    $(\zeta^\varepsilon_k, v^\varepsilon_k)_k$ converges in $\mathcal{C} ([0, T],  \mathcal{H}_{\delta}^{\ell-\frac{1}{2}} \times \mathcal{H}_{\delta}^{\ell+\frac{1}{2}}) \cap L^2 ((0, T), \mathcal{H}_{\delta}^{\ell} \times \mathcal{H}_{\delta}^{\ell+1})$ to some $(\zeta^\varepsilon, v^\varepsilon)$ solution to \eqref{sys:euler_modified_semicla}, which also satisfies
    \begin{equation*}
        \norm{\zeta^\varepsilon (t)}_{\ell, \delta}^2 + \norm{v^\varepsilon (t)}_{\ell + 1, \delta}^2 + 4 M_2 \int_0^t \Bigl( \norm{\zeta^\varepsilon (\tau)}_{\ell + \frac{1}{2}, \delta}^2 + \norm{v^\varepsilon (\tau)}_{\ell + \frac{3}{2}, \delta}^2 \Bigr) \diff \tau \leq 2 \omega_\textnormal{in}.
    \end{equation*}
\end{cor}

\subsection{Uniqueness of the solution}

We have just proved the existence of a solution to \eqref{sys:euler_modified_semicla}. We now prove the uniqueness of this solution thanks to similar estimates.

\begin{lem}
    The solution $(\zeta^\varepsilon, v^\varepsilon)$ to \eqref{sys:euler_modified_semicla} in $\mathcal{C} ([0, T],  \mathcal{H}_{\delta}^{\ell-\frac{1}{2}} \times \mathcal{H}_{\delta}^{\ell+\frac{1}{2}}) \cap L^\infty ((0, T),  \mathcal{H}_{\delta}^{\ell} \times \mathcal{H}_{\delta}^{\ell+1}) \cap L^2 ((0, T), \mathcal{H}_{\delta}^{\ell+\frac{1}{2}} \times \mathcal{H}_{\delta}^{\ell+\frac{3}{2}})$ is unique.
\end{lem}

\begin{proof}
Let $(\tilde{\zeta}^\varepsilon, \tilde{v}^\varepsilon)$ be another solution to \eqref{sys:euler_modified_semicla} in the space $L^\infty ((0, T), \mathcal{H}_{\delta}^\ell \times \mathcal{H}_{\delta}^{\ell+1}) \cap \mathcal{C} ([0, T], \mathcal{H}_{\delta}^{\ell - 1} \times \mathcal{H}_{\delta}^{\ell}) \cap  \cap L^2 ((0, T), \mathcal{H}_{\delta}^{\ell+\frac{1}{2}} \times \mathcal{H}_{\delta}^{\ell+\frac{3}{2}})$. In a similar way as in Section \ref{subsec:wellpos_scheme}, defining
\begin{gather*}
    \tilde{\omega}^\varepsilon (t) \coloneqq \norm{\tilde{\zeta}^\varepsilon}_{\infty, t, \ell,\delta}^2 + \norm{\tilde{v}^\varepsilon}_{\infty, t, \ell+1,\delta}^2, \\
    \tilde{\eta}^\varepsilon (t) \coloneqq \norm{\tilde{\zeta}^\varepsilon}_{2, t, \ell+\frac{1}{2},\delta}^2 + \norm{\tilde{v}^\varepsilon}_{2, t, \ell+\frac{3}{2},\delta}^2,
\end{gather*}
there holds for all $t \in [0, T]$
\begin{equation*}
    \mathcal{E}_{M, \ell, \delta} (\tilde{\zeta}^\varepsilon) (t) + \mathcal{E}_{M, \ell + 1, \delta} (\tilde{v}^\varepsilon) (t) \leq \omega_\textnormal{in} + (2 C_{\ell} (\sqrt{\tilde{\omega}^\varepsilon} + \tilde{\omega}^\varepsilon) + 2 \abs{\lambda} + 1 + 2 C_\ell ) \, \tilde{\eta}^\varepsilon.
\end{equation*}
Moreover, also from Lemma \ref{lem:toolbox}, $\norm{\tilde{\zeta}^\varepsilon (t)}_{\ell, \delta}^2 + \norm{\tilde{v}^\varepsilon (t)}_{\ell + 1, \delta}^2$ is continuous in time, and then so is $\tilde{\omega}^\varepsilon (t)$.
Therefore, the same kind of argument as already used along with a bootstrap property yields
\begin{equation} \label{eq:est_uniqueness}
    \mathcal{E}_{\abs{\lambda}, \ell, \delta} (\tilde{\zeta}^\varepsilon) (t) + \mathcal{E}_{\abs{\lambda}, \ell + 1, \delta} (\tilde{v}^\varepsilon) (t) \leq \omega_\textnormal{in}.
\end{equation}
Set now $\tilde{Z}^\varepsilon \coloneqq \tilde{\zeta}^\varepsilon - \zeta^\varepsilon$ and $\tilde{V}^\varepsilon \coloneqq \tilde{v}^\varepsilon - v^\varepsilon$.
Then, we obtain:
\begin{System} \notag
    \partial_t \tilde{V}^\varepsilon + (\tilde{v}^\varepsilon \cdot \nabla) \tilde{V}^\varepsilon + (\tilde{V}^\varepsilon \cdot \nabla) v^\varepsilon + \lambda \, \Re \tilde{Z}^\varepsilon = 0, \\
    \partial_t \tilde{Z}^\varepsilon + \nabla \Bigl( \tilde{V}^\varepsilon \cdot \tilde{\zeta}^\varepsilon \Bigr) + \nabla \Bigl( v^\varepsilon \cdot \tilde{Z}^\varepsilon \Bigr) + \nabla \divg \tilde{V}^\varepsilon = i \frac{\varepsilon}{2} \Bigl( \nabla \divg \tilde{Z}^\varepsilon + 2 \, \nabla ( \tilde{Z}^\varepsilon \cdot \tilde{\zeta}^\varepsilon ) + 2 \, \nabla ( \zeta^\varepsilon \cdot \tilde{Z}^\varepsilon ) \Bigr).
\end{System}
In the same way as in Section \ref{subsec:conv_scheme}, defining
\begin{gather*}
    \tilde{\Omega}^\varepsilon (t) \coloneqq \norm{\tilde{Z}^\varepsilon}_{\infty, t, \ell - \frac{1}{2},\delta}^2 + \norm{\tilde{V}^\varepsilon}_{\infty, t, \ell + \frac{1}{2},\delta}^2, \\
    \tilde{N}^\varepsilon (t) \coloneqq \norm{\tilde{Z}^\varepsilon}_{2, t, \ell,\delta}^2 + \norm{\tilde{V}^\varepsilon}_{2, t, \ell + 1,\delta}^2,
\end{gather*}
Lemma \ref{lem:toolbox} with \eqref{eq:tool1}-\eqref{eq:tool3} and \eqref{eq:tool5}-\eqref{eq:tool6} yields for all $t \in [0, T]$
\begin{equation*}
    \mathcal{E}_{M, \ell, \delta} (\tilde{Z}^\varepsilon) (t) + \mathcal{E}_{M, \ell + 1, \delta} (\tilde{V}^\varepsilon) (t) \leq ( 2 K^\ell \sqrt{\omega_\textnormal{in}} + K^\ell \omega_\textnormal{in} (2 + \varepsilon) + 1 + 2 \abs{\lambda} )
        \tilde{N}^\varepsilon + (2 \abs{\lambda} + K^\ell (1 + 2 \varepsilon)) \, T \, \tilde{\Omega}^\varepsilon.
\end{equation*}
From the definition of $M$, we get
\begin{equation*}
    \tilde{\Omega}^\varepsilon (t) \leq (2 \abs{\lambda} + K^\ell (1 + 2 \varepsilon)) \, T \, \tilde{\Omega}^\varepsilon (t),
\end{equation*}
which gives the conclusion as soon as $T$ is small enough so that $(2 \abs{\lambda} + 3 K^\ell ) \, T < 1$. This gives local uniqueness, which is sufficient to prove it even if $T$ is larger.
\end{proof}

\begin{rem}
    In particular, the previous computation \eqref{eq:est_uniqueness} and the uniqueness property give another estimate, slightly better than that of Lemma \ref{lem:est_scheme} for the $L^\infty( [0, T], \mathcal{H}^\ell_\delta \times \mathcal{H}^{\ell + 1}_\delta)$ norm, given in the next lemma.
\end{rem}

\begin{lem} \label{lem:est_scheme_2}
    For all $\varepsilon \in [0,1]$, there holds for all $t \in [0, T]$
    \begin{equation*}
        \mathcal{E}_{\abs{\lambda}, \ell, \delta} (\tilde{\zeta}^\varepsilon) (t) + \mathcal{E}_{\abs{\lambda}, \ell + 1, \delta} (\tilde{v}^\varepsilon) (t) \leq \omega_\textnormal{in}.
    \end{equation*}
\end{lem}

\section{Semiclassical limit} \label{sec:semicla}

We now address the semiclassical limit $\varepsilon \rightarrow 0$ in $(\zeta^\varepsilon, v^\varepsilon)$ variables, \textit{i.e.} the proof of the second part of Theorem \ref{th:main_th}. For this, we set $Z^\varepsilon \coloneqq \zeta^\varepsilon - \zeta^0$ and $V^\varepsilon \coloneqq v^\varepsilon - v^0$. Using \eqref{eq:conv_term_v_k}, they satisfy 
\begin{System} \label{sys:semicla1}
    \partial_t V^\varepsilon + \frac{1}{2} \nabla (V^\varepsilon \cdot v^\varepsilon) + \frac{1}{2} \nabla (v^0 \cdot V^\varepsilon) + \lambda \, \Re Z^\varepsilon = 0, \qquad V^\varepsilon (0) = \nabla v^\varepsilon_\textnormal{in} - \nabla v^0_\textnormal{in} \\
    \partial_t Z^\varepsilon + \nabla \Bigl( V^\varepsilon \cdot \zeta^\varepsilon \Bigr) + \nabla \Bigl( v^0 \cdot Z^\varepsilon \Bigr) + \nabla \divg V^\varepsilon = i \frac{\varepsilon}{2} \Bigl( \nabla \divg \zeta^\varepsilon + 2 \, \nabla ( \zeta^\varepsilon \cdot \zeta^\varepsilon ) \Bigr), \quad Z^\varepsilon (0) = \nabla \psi^\varepsilon_\textnormal{in} - \nabla \psi^0_\textnormal{in}.
\end{System}

\subsection{First case}

Set
\begin{gather*}
    \Omega^\varepsilon \coloneqq \norm{Z^\varepsilon}_{\infty, t, \ell - \frac{1}{2},\delta}^2 + \norm{V^\varepsilon}_{\infty, t, \ell + \frac{1}{2},\delta}^2.
\end{gather*}
By applying Lemma \ref{lem:toolbox} with $m = \ell - \frac{1}{2}$, $\theta_1 = 0$, $\theta_3 = 2 \theta_2 = i \varepsilon$, $\tilde{g}_1 = \tilde{g}_2 = \zeta^\varepsilon$ and \eqref{eq:tool3}-\eqref{eq:tool6}, using Lemma \ref{lem:est_scheme_2} and the following computations:
\begin{gather*}
    2 \, \norm{V^\varepsilon}_{2, t, \ell + 1, \delta} \norm{Z^\varepsilon}_{2, t, \ell, \delta} \leq \norm{V^\varepsilon}_{2, t, \ell + 1, \delta}^2 + \norm{Z^\varepsilon}_{2, t, \ell, \delta}^2, \\
    \norm{Z^\varepsilon}_{2, t, \ell, \delta} \norm{V^\varepsilon}_{2, t, \ell, \delta} \norm{\zeta^\varepsilon}_{\infty, t, \ell, \delta} \leq \frac{T}{2} \Omega^\varepsilon + \frac{1}{2} \omega_\textnormal{in} \norm{Z^\varepsilon}_{\infty, t, \ell, \delta}^2, \\
    \norm{v^0}_{\infty, t, \ell, \delta} \norm{Z^\varepsilon}_{2, t, \ell, \delta}^2 \leq \sqrt{\omega_\textnormal{in}} \norm{Z^\varepsilon}_{2, t, \ell, \delta}^2, \\
    \norm{Z^\varepsilon}_{2, t, \ell + \frac{1}{2}, \delta} \leq \norm{\zeta^\varepsilon}_{2, t, \ell + \frac{1}{2}, \delta} + \norm{\zeta^0}_{2, t, \ell + \frac{1}{2}, \delta} \leq \sqrt{\frac{2 \omega_\textnormal{in}}{\abs{\lambda}}},
\end{gather*}
we get for all $\varepsilon \leq 1$ and $t \in [0, T]$:
\begin{equation} \label{eq:semicla_est_compl}
    \mathcal{E}_{2 M_2, \ell - \frac{1}{2}, \delta} (Z^\varepsilon) (t) + \mathcal{E}_{2 M_2, \ell + \frac{1}{2}, \delta} (V^\varepsilon) (t) \leq (D^\varepsilon_{\ell - \frac{1}{2}})^2 + K^\ell T \Omega^\varepsilon (t) + \varepsilon \Bigl( \sqrt{2} K^\ell \omega_\textnormal{in}^\frac{3}{2} + \frac{\omega_\textnormal{in}}{\abs{\lambda}} \Bigr).
\end{equation}
Thus, up to taking for instance $T < \frac{K^\ell}{2}$, we obtain
\begin{equation*}
    \mathcal{E}_{M_2, \ell - \frac{1}{2}, \delta} (Z^\varepsilon) (t) + \mathcal{E}_{M_2, \ell + \frac{1}{2}, \delta} (V^\varepsilon) (t) \leq C \Bigl( (D^\varepsilon_{\ell - \frac{1}{2}})^2 + \varepsilon \Bigr).
\end{equation*}
This gives the conclusion for the first case.

\subsection{Case $\ell > \frac{d+1}{2}$}

Set now
\begin{gather*}
    \Omega^\varepsilon (t) \coloneqq \norm{Z^\varepsilon}_{\infty, t, \ell - 1,\delta}^2 + \norm{V^\varepsilon}_{\infty, t, \ell,\delta}^2.
\end{gather*}
Here, we use Lemma \ref{lem:toolbox} in a similar way as previously. However, we treat the term $i \frac{\varepsilon}{2} \nabla \divg \zeta^\varepsilon$ with \eqref{eq:tool6} instead of with $\tilde{g}_1$, and the term $i \varepsilon \, \nabla ( \zeta^\varepsilon \cdot \zeta^\varepsilon )$ with $\tilde{g}_2$ instead of with \eqref{eq:tool7}. Then, we also use the following computations:
\begin{gather*}
    \varepsilon \norm{Z^\varepsilon}_{2, t, \ell - \frac{1}{2}, \delta} \norm{\zeta^\varepsilon}_{2, t, \ell + \frac{1}{2}, \delta} \leq C_\ell \, \omega_\textnormal{in} \, \norm{Z^\varepsilon}_{2, t, \ell - \frac{1}{2}, \delta}^2 + \varepsilon^2 \frac{1}{8 \abs{\lambda} C_\ell}, \\
    \varepsilon \norm{Z^\varepsilon}_{2, t, \ell - 1, \delta} \norm{\zeta^\varepsilon}_{\infty, t, \ell, \delta}^2 \leq \frac{1}{2} \norm{Z^\varepsilon}_{\infty, t, \ell - 1, \delta}^2 + \frac{\varepsilon^2}{2} T \omega_\textnormal{in}^4.
\end{gather*}
Thus, we get
\begin{equation*}
    \mathcal{E}_{2 M_2, \ell - 1, \delta} (Z^\varepsilon) (t) + \mathcal{E}_{2 M_2, \ell, \delta} (V^\varepsilon) (t) \leq (D_{\ell - 1}^\varepsilon)^2 + \frac{1}{2} \Omega^\varepsilon (t) + \frac{\varepsilon^2}{4} \Bigl( T^2 \, K^\ell \omega_\textnormal{in}^4 + \frac{1}{2 \abs{\lambda} \, C_\ell} \Bigr),
\end{equation*}
which gives the conclusion for the second statement of the second part of Theorem \ref{th:main_th}.

\begin{rem}
    The fact that we cannot recover the $O (\varepsilon^2)$ in the first part comes from the term $\nabla \divg \zeta^\varepsilon$.
    Indeed, the highest regularity for which we have an estimate for $\zeta^\varepsilon$ is the $\mathcal{H}_\delta^{\ell + \frac{1}{2}}$, which is $L^2$ in time.
    Therefore, when one wants to estimate $\abs{\langle Z^\varepsilon, \nabla \divg \zeta^\varepsilon \rangle_{\ell - \frac{1}{2}, \delta}}$ for Lemma \ref{lem:prop_anal_norm_diff} (or Lemma \ref{lem:toolbox}), since we cannot go further $\ell + \frac{1}{2}$ in the norm of $\zeta^\varepsilon$, there must be at least $\ell + \frac{1}{2}$ for the norm of $Z^\varepsilon$, which is not very optimal for this estimate: we would want at most $\ell$.
    This problem does not occur for the second case because we estimate the $(\ell - 1, \delta)$ scalar product: it allows an extra notch backwards for the regularity of $Z^\varepsilon$, which is sufficient for falling into a better framework for our estimates.
\end{rem}

\section{Properties on $\psi^\varepsilon$ and $\phi^\varepsilon$} \label{sec:cor_proof}

In this section, we prove the second and third parts of Corollary \ref{cor:cauchy}.

\subsection{Behavior at infinity}

\subsubsection{Analyticity of $\partial_t \psi^\varepsilon$}

The first part of this proof is a result of analyticity of $\partial_t \psi^\varepsilon$.

\begin{lem} \label{lem:anal_dt_psi}
    There holds $\partial_t \psi^\varepsilon \in L^2_T \mathcal{H}^{\ell - \frac{1}{2}}_\delta$.
\end{lem}

\begin{proof}
    We know from \eqref{eq:rel_v_phi_3} that
    \begin{equation*}
        \partial_t \psi^\varepsilon + v^\varepsilon \cdot \zeta^\varepsilon + \divg v^\varepsilon = i \frac{\varepsilon}{2} \Bigl( \divg \zeta^\varepsilon + 2 \, \zeta^\varepsilon \cdot \zeta^\varepsilon \Bigr).
    \end{equation*}
    Moreover, using Lemma \ref{lem:est_scheme_2}, there holds
    \begin{equation*}
        \norm{v^\varepsilon \cdot \zeta^\varepsilon}_{\ell, \delta} \leq K^\ell \norm{v^\varepsilon}_{\ell, \delta} \norm{\zeta^\varepsilon}_{\ell, \delta} \leq K^\ell \omega_\textnormal{in},
    \end{equation*}
    \begin{equation*}
        \norm{\divg v^\varepsilon}_{\ell, \delta} \leq \norm{v^\varepsilon}_{\ell + 1, \delta} \leq \sqrt{\omega_\textnormal{in}},
    \end{equation*}
    \begin{equation*}
        \norm{\zeta^\varepsilon \cdot \zeta^\varepsilon}_{\ell, \delta} \leq K^\ell \norm{\zeta^\varepsilon}_{\ell, \delta}^2 \leq K^\ell \omega_\textnormal{in},
    \end{equation*}
    \begin{equation*}
        \norm{\divg \zeta^\varepsilon}_{\ell - \frac{1}{2}} \leq \norm{\zeta^\varepsilon}_{\ell + \frac{1}{2}, \delta}.
    \end{equation*}
    The latter is $L^2$ in time (in $(0, T)$) by using also Lemma \ref{lem:est_scheme_2}. Therefore, the conclusion easily follows.
\end{proof}

\subsubsection{$\psi^\varepsilon$ case} \label{subsec:psi_case}

Now, we prove the case of $\psi^\varepsilon$.
Setting $\Psi^\varepsilon \coloneqq \psi^\varepsilon - \psi_\textnormal{in}^\varepsilon$, Lemma \ref{lem:anal_dt_psi} (along with the facts that $\Psi^\varepsilon (0) = 0$ and $\partial_t \Psi^\varepsilon = \partial_t \psi^\varepsilon$) yields
\begin{equation*}
    \Psi^\varepsilon \in H^1 ((0, T), \mathcal{H}^{\ell - \frac{1}{2}}_\delta).
\end{equation*}
In particular, $\psi^\varepsilon - \psi_\textnormal{in}^\varepsilon \in \mathcal{C} ([0, T], \mathcal{H}^{\ell - \frac{1}{2}}_\delta)$.
But we also have $\nabla \psi^\varepsilon \in \mathcal{C} ([0, T], \mathcal{H}^{\ell - \frac{1}{2}}_\delta)$, and it is of course also the same for $\nabla \psi_\textnormal{in}^\varepsilon$.
Therefore, by \eqref{eq:anal_deriv}, we obtain
\begin{equation*}
    \psi^\varepsilon - \psi_\textnormal{in}^\varepsilon \in \mathcal{C}_T \mathcal{H}^{\ell + \frac{1}{2}}_\delta.
\end{equation*}
Moreover, Lemma \ref{lem:toolbox} gives for all $t \in [0, T]$
\begin{equation*}
    \mathcal{E}_{M, \ell, \delta} (\Psi^\varepsilon) (t) \leq \norm{\Psi^\varepsilon}_{2, t, \ell + \frac{1}{2}, \delta} \norm{\partial_t \Psi^\varepsilon}_{2, t, \ell - \frac{1}{2}, \delta}.
\end{equation*}
Applying Lemma \ref{lem:anal_dt_psi} once again yields
\begin{equation*}
    \Psi^\varepsilon \in L^\infty_T \mathcal{H}^{\ell + 1}_{\delta} \cap L^2_T \mathcal{H}^{\ell + \frac{3}{2}}_{\delta}.
\end{equation*}
Thus, we get \eqref{eq:1st_point_main_cor}.

\subsubsection{$\phi^\varepsilon$ case}

In a similar way, we now prove an analyticity result for $\partial_t \Phi^\varepsilon$ where
\begin{equation*}
    \Phi^\varepsilon \coloneqq \phi^\varepsilon - \phi_\textnormal{in}^\varepsilon - \lambda t \psi_\textnormal{in}^\varepsilon.
\end{equation*}

\begin{lem} \label{lem:anal_dt_phi}
    There holds $\partial_t \Phi^\varepsilon \in L^\infty_T \mathcal{H}^{\ell + 1}_{\delta} \cap L^2_T \mathcal{H}^{\ell + \frac{3}{2}}_{\delta}$.
\end{lem}

\begin{proof}
    We know that
    \begin{equation*}
        \partial_t \phi^\varepsilon + \frac{1}{2} \abs{v^\varepsilon}^2 + \lambda \Re \psi^\varepsilon = 0.
    \end{equation*}
    Therefore, we have
    \begin{equation*}
        \partial_t \Phi^\varepsilon + \frac{1}{2} \abs{v^\varepsilon}^2 + \lambda \Re \Psi^\varepsilon = 0.
    \end{equation*}
    From the previous result and from the fact that
    \begin{equation*}
        \norm{\abs{v^\varepsilon}^2}_{\ell + \frac{3}{2}, \delta} \leq K^{\ell + \frac{3}{2}, \ell + 1} \norm{v^\varepsilon}_{\ell + \frac{3}{2}, \delta} \norm{v^\varepsilon}_{\ell + 1, \delta} \leq K^{\ell + \frac{3}{2}, \ell + 1} \sqrt{\omega_\textnormal{in}} \norm{v^\varepsilon}_{\ell + \frac{3}{2}, \delta},
    \end{equation*}
    along with Lemma \ref{lem:est_scheme_2}, this yields the conclusion.
\end{proof}

The proof of the second statement of the second part of Corollary \ref{cor:cauchy} is then similar as in Section \ref{subsec:psi_case}.

\subsection{Semiclassical limit}

Now, we prove the results of the semiclassical limit for $\psi^\varepsilon$ and $\phi^\varepsilon$.
Set
\begin{equation*}
    P^\varepsilon = \psi^\varepsilon - \psi^0, \qquad
    Q^\varepsilon = \phi^\varepsilon - \phi^0.
\end{equation*}
Then, there holds
\begin{gather*}
    \partial_t P^\varepsilon + V^\varepsilon \cdot \zeta^\varepsilon + v^0 \cdot Z^\varepsilon + \divg V^\varepsilon = i \frac{\varepsilon}{2} \Bigl( \divg \zeta^\varepsilon + 2 \, \zeta^\varepsilon \cdot \zeta^\varepsilon \Bigr), \\
    \partial_t Q^\varepsilon + \frac{1}{2} V^\varepsilon \cdot v^\varepsilon + \frac{1}{2} v^0 \cdot V^\varepsilon + \lambda \Re P^\varepsilon = 0.
\end{gather*}

\subsubsection{First case}

By applying Lemma \ref{lem:toolbox}, \eqref{eq:tool9} and \eqref{eq:tool10} with $m = \ell - \frac{1}{2}$ and $g = P^\varepsilon$, and using the second part of Theorem \ref{th:main_th} and Lemma \ref{lem:est_scheme_2}, we obtain for all $t \in [0, T]$
\begin{equation*}
    \mathcal{E}_{M, \ell - \frac{1}{2}, \delta} (P^\varepsilon) (t) \leq \norm{P^\varepsilon (0)}_{\ell - \frac{1}{2}, \delta_\textnormal{in}}^2 + C \biggl[ \varepsilon + \Bigl( \varepsilon + (D_{\ell - \frac{1}{2}}^\varepsilon)^2 \Bigr)^\frac{1}{2} \biggr] \norm{P^\varepsilon}_{2, t, \ell, \delta},
\end{equation*}
for some constant $C > 0$. Therefore we get for all $\varepsilon \in [0, 1]$
\begin{equation*}
    \mathcal{E}_{2 M_2, \ell - \frac{1}{2}, \delta} (P^\varepsilon) (t) \leq \norm{P^\varepsilon (0)}_{\ell - \frac{1}{2}, \delta_\textnormal{in}}^2 + C \Bigl( \varepsilon + (D_{\ell - \frac{1}{2}}^\varepsilon)^2 \Bigr).
\end{equation*}
Then, for $Q^\varepsilon$, we apply again Lemma \ref{lem:toolbox} with $f = Q^\varepsilon$, \eqref{eq:tool0} and \eqref{eq:tool3}, so that we also get:
\begin{equation*}
    \mathcal{E}_{M, \ell + \frac{1}{2}, \delta} (Q^\varepsilon) (t) \leq \norm{Q^\varepsilon (0)}_{\ell + \frac{1}{2}, \delta}^2 + C \Bigl( \varepsilon + (D_{\ell - \frac{1}{2}}^\varepsilon)^2  \Bigr)^\frac{1}{2} \norm{Q^\varepsilon (\tau)}_{2, t, \ell + 1, \delta}.
\end{equation*}
Therefore we have in the same way
\begin{equation*}
    \mathcal{E}_{2 M_2, \ell + \frac{1}{2}, \delta} (Q^\varepsilon) (t) \leq \norm{Q^\varepsilon (0)}_{\ell + \frac{1}{2}, \delta}^2 + C \Bigl( \varepsilon + (D_{\ell - \frac{1}{2}}^\varepsilon)^2 \Bigr).
\end{equation*}
Hence, there holds
\begin{equation*}
    \mathcal{E}_{2 M_2, \ell - \frac{1}{2}, \delta} (P^\varepsilon) (t) + \mathcal{E}_{2 M_2, \ell + \frac{1}{2}, \delta} (Q^\varepsilon) (t)
        \leq (\tilde{D}_{\ell - \frac{1}{2}}^\varepsilon)^2 + C ( \varepsilon + (D_{\ell - \frac{1}{2}}^\varepsilon)^2 ).
\end{equation*}
\eqref{eq:anal_deriv} yields
\begin{equation*}
    (\tilde{D}_{\ell - \frac{1}{2}}^\varepsilon)^2 + (D_{\ell - \frac{1}{2}}^\varepsilon)^2 = (\tilde{D}_{\ell + \frac{1}{2}}^\varepsilon)^2,
\end{equation*}
so that
\begin{equation*}
    \mathcal{E}_{2 M_2, \ell - \frac{1}{2}, \delta} (P^\varepsilon) (t) + \mathcal{E}_{2 M_2, \ell + \frac{1}{2}, \delta} (Q^\varepsilon) (t)
        \leq C \Bigl( \varepsilon + (\tilde{D}_{\ell + \frac{1}{2}}^\varepsilon)^2 \Bigr).
\end{equation*}
The conclusion of the first statement then comes from the previous computation and the second part of Theorem \ref{th:main_th} along with \eqref{eq:anal_deriv}.

\subsubsection{Case $\frac{d+1}{2} < \ell$} \label{subsec:semicla2}

For this case, applying again Lemma \ref{lem:toolbox} with $g = P^\varepsilon$ but with $m = \ell - 1 > \frac{d-1}{2}$, \eqref{eq:tool8} and \eqref{eq:tool10}, the estimates give for some $C > 0$:
\begin{equation*}
    \mathcal{E}_{M, \ell - 1, \delta} (P^\varepsilon) (t) \leq \norm{P^\varepsilon (0)}_{\ell - 1, \delta}^2 + C ( \varepsilon^2 + (D_{\ell - 1}^\varepsilon)^2 )^\frac{1}{2} \norm{P^\varepsilon}_{2, \ell - \frac{1}{2}, \delta},
\end{equation*}
so that we obtain
\begin{equation*}
    \mathcal{E}_{2 M_2, \ell - 1, \delta} (P^\varepsilon) (t) \leq (\tilde{D}^\varepsilon_{\ell - 1})^2 + C ( \varepsilon^2 + (D_{\ell - 1}^\varepsilon)^2 ).
\end{equation*}
Then, for $Q^\varepsilon$, we apply Lemma \ref{lem:toolbox}, \eqref{eq:tool0} and \eqref{eq:tool3} with $f = Q^\varepsilon$, which gives:
\begin{equation*}
    \mathcal{E}_{M, \ell, \delta} (Q^\varepsilon) (t) \leq \norm{Q^\varepsilon (0)}_{\ell, \delta}^2 + C ( \varepsilon^2 + (D_{\ell - 1}^\varepsilon)^2 )^\frac{1}{2} \, \norm{Q^\varepsilon}_{2, \ell + \frac{1}{2}, \delta},
\end{equation*}
and the conclusion follows in the same way as in the previous case.

\section{Semiclassical limit of the wave function} \label{sec:semicla_lim}

In this section, we address the semiclassical limit of the wave function $u^\varepsilon = e^{\frac{\psi^\varepsilon}{2} + \frac{\phi^\varepsilon}{\varepsilon}}$ with Theorem \ref{th:semicla_th}. We also prove Lemma \ref{lem:semicla_cond}. But first, we need to address the Cauchy problem of \eqref{sys:euler_modified_semicla2}.

\subsection{Cauchy problem of \eqref{sys:euler_modified_semicla2}}

\begin{theorem}
    Let $\frac{d-1}{2} < m \leq l-1$. For any $(\nabla \psi_{\textnormal{in}, 1}, \nabla \phi_{\textnormal{in}, 1}) \in \mathcal{H}_{\delta_\textnormal{in}}^m \times \mathcal{H}_{\delta_\textnormal{in}}^{m+1}$, there exists a unique $(\zeta_1, v_1) \in (L^\infty_T \mathcal{H}_{\delta}^m \times L^\infty_T \mathcal{H}_{\delta}^{m+1}) \cap (L^2_T \mathcal{H}_{\delta}^{m+\frac{1}{2}} \times L^2_T \mathcal{H}_{\delta}^{m+\frac{3}{2}}) \cap (\mathcal{C}_T \mathcal{H}_{\delta}^{m-\frac{1}{2}} \times \mathcal{C}_T \mathcal{H}_{\delta}^{m+\frac{1}{2}})$ solution to \eqref{sys:euler_modified_semicla2}.
\end{theorem}

\begin{proof}
    The proof is rather similar as the Cauchy theory of \eqref{sys:euler_modified_semicla} developed in Section \ref{sec:Cauchy}, and we present here the main steps. First, for the existence, we define a scheme:
    \begin{itemize}
    \item $\zeta_{1, 0} (t) \coloneqq \nabla \psi_{\textnormal{in}, 1}$ and $v_{1,0} (t) \coloneqq \nabla \phi_{\textnormal{in}, 1}$ for all $t \geq 0$,
    \item For any $k \in \mathbb{N}$, $(\zeta_{1,k+1}, v_{1, k+1})$ is defined by
    \begin{System} \notag
        \partial_t v_{1, k+1} + \nabla ( v^0 \cdot v_{1, k+1} ) + \lambda \, \Re \zeta_{1, k} = 0, \\
        \partial_t \zeta_{1, k+1} + \nabla ( v^0 \cdot \zeta_{1,k+1} ) + \nabla ( v_{1,k+1} \cdot \zeta^0) + \nabla \divg v_{1, k+1} = \frac{i}{2} \Bigl( \nabla \divg \zeta^0 + 2 \, \nabla ( \zeta^0 \cdot \zeta^0 ) \Bigr),
    \end{System}
    and the initial data $\zeta_{1,k+1} (0) = \nabla \psi_{\textnormal{in}, 1}$ and $v_{1, k+1} (0) = \nabla \phi_{\textnormal{in}, 1}$.
    \end{itemize}
    The previous system and Lemma \ref{lem:toolbox} give the following estimates: for all $t \in [0, T]$ and $k \in \mathbb{N}$,
    \begin{equation*}
        \mathcal{E}_{2 M_2, m, \delta} (\zeta_{1,k+1}) (t) + \mathcal{E}_{2 M_2, m, \delta} (v_{1,k+1}) (t) \leq \abs{\lambda} \tnorm{\zeta_{1,k}}_{2, T, m+\frac{1}{2}, \delta} + \frac{1}{2} \sqrt{\omega_\textnormal{in}} + K^{\ell} \omega_\textnormal{in}.
    \end{equation*}
    From this estimate, one can prove by induction that, for all $k \in \mathbb{N}$ and $t \in [0, T]$,
    \begin{equation*}
        \mathcal{E}_{2 M_2, m, \delta} (\zeta_{1,k}) (t) + \mathcal{E}_{2 M_2, m, \delta} (v_{1,k}) (t) \leq \sqrt{\omega_\textnormal{in}} + 2 K^{\ell} \omega_\textnormal{in}.
    \end{equation*}
    Moreover, there also holds
    \begin{equation*}
        \mathcal{E}_{2 M_2, m, \delta} (\zeta_{1,k+1} - \zeta_{1, k}) (t) + \mathcal{E}_{2 M_2, m, \delta} (v_{1,k+1} - v_{1,k}) (t) \leq \abs{\lambda} \tnorm{\zeta_{1,k} - \zeta_{1, k-1}}_{2, T, m+\frac{1}{2}, \delta},
    \end{equation*}
    which proves the convergence of the scheme in $(\mathcal{C}_T \mathcal{H}^m_\delta \times \mathcal{C}_T \mathcal{H}^{m+1}_\delta) \cap (L^2_T \mathcal{H}^{m + \frac{1}{2}}_\delta \times L^2_T \mathcal{H}^{m+\frac{3}{2}}_\delta)$ like in Section \ref{subsec:conv_scheme}. The uniqueness can also be proved with a similar computation.
\end{proof}

\subsection{Proof of Theorem \ref{th:semicla_th}}

We recall that $Z^\varepsilon = \zeta^\varepsilon - z^0$ and $V^\varepsilon = v^\varepsilon - v^0$. Define $\mathcal{Z}^\varepsilon \coloneqq Z^\varepsilon - \varepsilon \zeta_1$ and $\mathcal{V}^\varepsilon \coloneqq  V^\varepsilon - \varepsilon v_1$. Then we have
\begin{gather*}
    \partial_t \mathcal{V}^\varepsilon + \nabla (v^0 \cdot \mathcal{V}^\varepsilon) + \lambda \, \Re \mathcal{Z}^\varepsilon = \frac{1}{2} \nabla (V^\varepsilon \cdot V^\varepsilon), \\
    \partial_t \mathcal{Z}^\varepsilon + \nabla \Bigl( \mathcal{V}^\varepsilon \cdot \zeta^0 \Bigr) + \nabla \Bigl( v^0 \cdot \mathcal{Z}^\varepsilon \Bigr) + \nabla \divg \mathcal{V}^\varepsilon = \nabla \Bigl( V^\varepsilon \cdot Z^\varepsilon \Bigr) + i \frac{\varepsilon}{2} \Bigl( \nabla \divg Z^\varepsilon + 2 \, \nabla ( Z^\varepsilon \cdot \zeta^0 ) + 2 \, \nabla ( Z^\varepsilon \cdot \zeta^\varepsilon ) \Bigr).
\end{gather*}
This system is very similar to \eqref{sys:semicla1} and the estimates are actually the same up to two differences. First, the source terms (at the right-hand side of each equation) are $O ( \varepsilon^2 )$ but only in $L^2_T \mathcal{H}^{\ell - \frac{3}{2}}_\delta$ for the first equation and in $L^2_T \mathcal{H}^{\ell - \frac{5}{2}}_\delta$ for the second one. Indeed, there holds
\begin{gather*}
    \tnorm{\nabla (V^\varepsilon \cdot V^\varepsilon)}_{2, T, \ell - \frac{3}{2}, \delta} \leq K^{\ell} \sqrt{T} \tnorm{V^\varepsilon}_{\infty, T, \ell, \delta}^2 \leq C \varepsilon^2, \\
    \tnorm{\nabla (V^\varepsilon \cdot Z^\varepsilon)}_{2, T, \ell - \frac{5}{2}, \delta} \leq K^{\ell - 1} \sqrt{T} \tnorm{V^\varepsilon}_{\infty, T, \ell - 1, \delta} \tnorm{Z^\varepsilon}_{\infty, T, \ell - 1, \delta} \leq C \varepsilon^2, \\
    \varepsilon \tnorm{\nabla \divg Z^\varepsilon}_{2, T, \ell - \frac{5}{2}, \delta} \leq \varepsilon \tnorm{Z^\varepsilon}_{2, T, \ell - \frac{1}{2}, \delta} \leq C \varepsilon^2, \\
    \varepsilon \tnorm{\nabla (Z^\varepsilon \cdot \zeta^0)}_{2, T, \ell - \frac{5}{2}, \delta} \leq \varepsilon K^{\ell - 1} \sqrt{T} \tnorm{Z^\varepsilon}_{\infty, T, \ell - 1, \delta} \tnorm{\zeta^0}_{\infty, T, \ell, \delta} \leq C \varepsilon^2, \\
    \varepsilon \tnorm{\nabla (Z^\varepsilon \cdot \zeta^\varepsilon)}_{2, T, \ell - \frac{5}{2}, \delta} \leq \varepsilon K^{\ell - 1} \sqrt{T} \tnorm{Z^\varepsilon}_{\infty, T, \ell - 1, \delta} \tnorm{\zeta^\varepsilon}_{\infty, T, \ell, \delta} \leq C \varepsilon^2.
\end{gather*}
Then, we apply Lemma \ref{lem:toolbox} with $m = \ell - 2$, and in particular \eqref{eq:tool5}. For this, we need $\ell - 2 > \frac{d-1}{2}$, which is assumed by Assumption \ref{ass:semicla1}. Therefore, we get for some $C > 0$ and for all $t \in [0, T]$
\begin{equation*}
    \mathcal{E}_{2 M_2, \ell - 2, \delta} (\mathcal{Z}^\varepsilon) + \mathcal{E}_{2 M_2, \ell - 1, \delta} (\mathcal{V}^\varepsilon) (t) \leq (r_{\ell - 2}^\varepsilon)^2 + C \, \varepsilon^4.
\end{equation*}
As for the case of $\psi^\varepsilon$ and $\phi^\varepsilon$, we also have $P^\varepsilon = \psi^\varepsilon - \psi^0$ and $Q^\varepsilon = \phi^\varepsilon - \phi^0$, so that by defining $\mathcal{P}^\varepsilon \coloneqq P^\varepsilon - \varepsilon \psi_1$ and $\mathcal{Q}^\varepsilon = Q^\varepsilon - \varepsilon \phi_1$, we obtain
\begin{gather*}
    \partial_t \mathcal{P}^\varepsilon + \mathcal{V}^\varepsilon \cdot \zeta^0 + v^0 \cdot \mathcal{Z}^\varepsilon + \divg \mathcal{V}^\varepsilon = V^\varepsilon \cdot Z^\varepsilon + i \frac{\varepsilon}{2} \Bigl( \divg Z^\varepsilon + 2 \, Z^\varepsilon \cdot \zeta^\varepsilon + 2 \, Z^\varepsilon \cdot \zeta^0 \Bigr), \\
    \partial_t \mathcal{Q}^\varepsilon + \mathcal{V}^\varepsilon \cdot v^0 + \lambda \Re P^\varepsilon = \frac{1}{2} V^\varepsilon \cdot V^\varepsilon.
\end{gather*}
Like in Section \ref{subsec:semicla2} but with $m = \ell - 2$, we get with the first equation
\begin{equation*}
    \mathcal{E}_{M, \ell - 2, \delta} (\mathcal{P}^\varepsilon) (t) \leq \norm{\mathcal{P}^\varepsilon (0)}_{\ell - 2, \delta}^2 + C ( \varepsilon^4 + (r_{\ell - 2}^\varepsilon)^2 )^\frac{1}{2} \norm{P^\varepsilon}_{2, \ell - \frac{3}{2}, \delta},
\end{equation*}
and with the second one
\begin{equation*}
    \mathcal{E}_{M, \ell - 1, \delta} (\mathcal{Q}^\varepsilon) (t) \leq \norm{\mathcal{Q}^\varepsilon (0)}_{\ell - 1, \delta}^2 + C ( \varepsilon^4 + (r_{\ell - 2}^\varepsilon)^2 )^\frac{1}{2} \, \norm{Q^\varepsilon}_{2, \ell - \frac{1}{2}, \delta},
\end{equation*}
and the conclusion easily follows, using the fact that
\begin{equation*}
    \norm{\mathcal{P}^\varepsilon (0)}_{\ell - 2, \delta}^2 + \norm{\mathcal{Q}^\varepsilon (0)}_{\ell - 1, \delta}^2 + (r_{\ell - 2}^\varepsilon)^2 \leq 2 (\tilde{r}_{\ell - 1}^\varepsilon)^2,
\end{equation*}
with \eqref{eq:anal_deriv}.

\subsection{Proof of Lemma \ref{lem:semicla_cond}}

We now address Lemma \ref{lem:semicla_cond}. First, assume $\phi_1 (t) \equiv 0$ on $[0, T]$, which yields not only $\phi_{\textnormal{in}, 1} \equiv 0$ but also $\nabla \phi_1 (t) \equiv 0$. Then, the first equation of \eqref{sys:wkb_riemann2} gives
\begin{equation*}
    \Re \psi_1 = 0.
\end{equation*}
This is in particular true for $t = 0$. Since $\psi_1 (0) = \psi_{\textnormal{in}, 1}$ is real-valued by assumption, we obtain $\psi_{\textnormal{in}, 1} \equiv 0$.

On the other hand, since $\phi_1$ is real-valued, taking the real part of the second equation of \eqref{sys:wkb_riemann2} leads to a system in $\Re \psi^1$ and $\phi^1$:
\begin{System} \notag
    \partial_t \phi_1 + v^0 \cdot \nabla \phi_1 + \lambda \Re \psi_1 = 0, \qquad \qquad &\phi_1 (0) = \phi_{\textnormal{in}, 1}, \\
    \partial_t \Re \psi_1 + v^0 \cdot \nabla \Re \psi_1 + \nabla \phi_1 \cdot \Re \zeta^0 + \Delta \phi_1 = 0, \qquad \qquad  &\Re \psi_1 (0) = \psi_{\textnormal{in}, 1}.
\end{System}
This system is linear in $(\Re \psi_1, \phi_1)$, without any source term.
Therefore, if $(\psi_{\textnormal{in}, 1}, \phi_{\textnormal{in}, 1}) \equiv (0, 0)$, we get $(\Re \psi_1, \phi_1) \equiv (0,0)$, which gives the conclusion.

\section{Assumptions on the initial data} \label{sec:ass_in_data}

In this section, we discuss about Assumption \ref{ass:bound} for the initial data. In particular, the analytic behavior is asked only for the gradient of the initial data $(\nabla \psi_\textnormal{in}^\varepsilon, \nabla \phi_\textnormal{in}^\varepsilon)$. One can show that this statement is different from asking the analyticity for the initial data $(\psi_\textnormal{in}^\varepsilon, \phi_\textnormal{in}^\varepsilon)$ or even $(\psi_\textnormal{in}^\varepsilon, \nabla \phi_\textnormal{in}^\varepsilon)$ directly. Indeed, when we consider the Fourier transform of these functions, we know that they are linked (for instance for $\psi_\textnormal{in}^\varepsilon$) through the relation
\begin{equation*}
    \mathcal{F} (\nabla \psi_\textnormal{in}^\varepsilon) = - i \xi \mathcal{F} (\psi_\textnormal{in}^\varepsilon).
\end{equation*}
In particular, when we consider analyticity, we multiply these Fourier transforms by some $e^{- \delta_\textnormal{in} \langle \xi \rangle}$ and ask them to be square integrable.
Thus, if $\mathcal{F} (\nabla \psi_\textnormal{in}^\varepsilon) e^{- \delta_\textnormal{in} \langle \xi \rangle}$ is $L^2$, then the previous relation gives that $\mathcal{F} (\psi_\textnormal{in}^\varepsilon) e^{- \delta_\textnormal{in} \langle \xi \rangle}$ is square integrable for $\abs{\xi} \rightarrow \infty$.
However, we could still have a problem at $\xi = 0$, for instance if $\abs{\mathcal{F} (\psi_\textnormal{in}^\varepsilon) (\xi)} \sim \abs{\xi}^{-\frac{d}{2}}$ which is not square integrable but which gives $\abs{\mathcal{F} (\nabla \psi_\textnormal{in}^\varepsilon) (\xi)}^2 \sim \abs{\xi}^{-d+1}$ which is integrable.

This problem at $\xi = 0$ is actually linked to the behavior of $\psi_\textnormal{in}^\varepsilon (x)$ for $\abs{x} \rightarrow \infty$ since we formally have
\begin{align*}
    \mathcal{F} (\nabla \psi_\textnormal{in}^\varepsilon) (0) &= \int \nabla \psi_\textnormal{in}^\varepsilon (x) \diff x \\
        &= \lim_{R \rightarrow \infty} \int_{\abs{x} \leq R} \nabla \psi_\textnormal{in}^\varepsilon (x) \diff x \\
        &= \lim_{R \rightarrow \infty} \int_{\abs{x} = R} \psi_\textnormal{in}^\varepsilon (y) \, \vec{n} \diff \sigma (y).
\end{align*}
In particular, in dimension $d = 1$, we show that we can have any possible limit at infinity, and even different limits at $\pm \infty$.

\begin{lem} \label{lem:ass_in_data}
    For any pair $(a_-, a_+) \in (\mathbb{R} \cup \{ \pm \infty \})^2$, there exists $f \in \mathcal{C}^1 (\mathbb{R})$ such that $f' \in \mathcal{H}_\delta^0$ for any $\delta > 0$ and $\lim_{\pm \infty} f = a_\pm$.
\end{lem}

\begin{proof}
    We will prove this result in three steps. First, we will prove it for $(a_-, a_+) \in \mathbb{R}^2$ by constructing a first function whose limit at $- \infty$ (resp. $+ \infty$) is $0$ (resp. $1$) and satisfying the previous regularity conditions. Then, we will prove it when exactly one of them is finite and the other infinity. Finally, we will prove it for two infinity limits.
    
    \textit{First step.} We construct here a first function which will be used to prove the case $(a_-, a_+) \in \mathbb{R}^2$. Define
    \begin{equation*}
        h_1 (x) \coloneqq \frac{1}{\sqrt{2 \pi}} e^{- \frac{x^2}{2}}.
    \end{equation*}
    It is known that $h_1 \in L^1 \cap L^2$ and
    \begin{equation*}
        \hat{h}_1 (\xi) = e^{- \frac{\xi^2}{2}}.
    \end{equation*}
    Therefore, $h_1 \in \mathcal{H}_\delta^0$ for any $\delta > 0$. Then, we define:
    \begin{equation*}
        g_1 (x) \coloneqq \int_{- \infty}^{x} h_1(y) \diff y.
    \end{equation*}
    It is well defined since $h_1 \in L^1$. Moreover, it is obviously $\mathcal{C}^1$ with $g_1' = h_1 \in \mathcal{H}_\delta^0$ for any $\delta > 0$, and $\lim_{- \infty} g_1 = 0$. Furthermore, it is also known that
    \begin{equation*}
        \lim_{+ \infty} g_1 = \int_{- \infty}^{+ \infty} h_1(y) \diff y = 1.
    \end{equation*}
    Then, for $(a_-, a_+) \in \mathbb{R}^2$, $f_1 \coloneqq a_- + (a_+ - a_-) \, g_1$ satisfies the needed assumptions.
    
    \textit{Second step.} We assume here $a_+ = + \infty$ and $a_- \in \mathbb{R}$. Define
    \begin{equation*}
        h_2 (x) \coloneqq 
        \begin{cases}
            (1+x)^{-1} \qquad &\textnormal{if } x > 0, \\
            0 \qquad &\textnormal{otherwise}.
        \end{cases}
    \end{equation*}
    We know that $h_2 \in L^2$, therefore $\hat{h}_2 \in L^2$ too. However, $h_2 \notin L^1$. More precisely, $h_2$ is not integrable at $+ \infty$. Then, we define $g_2$ by convolution:
    \begin{equation*}
        g_2 \coloneqq h_1 * h_2.
    \end{equation*}
    This is well defined pointwise since both $h_1$ and $h_2$ are in $L^2$, and it is also in $L^2$ since $h_1 \in L^1$.
    In particular, we have
    \begin{equation*}
        \hat{g}_2 (\xi) = \hat{h}_2 (\xi) \, e^{- \frac{\xi^2}{2}} \in L^2 (e^{-2 \delta \langle \xi \rangle} \diff \xi),
    \end{equation*}
    for every $\delta > 0$, therefore $g_2 \in \mathcal{H}_\delta^0$. In particular, $g_2 \in \mathcal{C}^\infty$. Moreover, $g_2 \geq 0$ from the fact that both $h_1$ and $h_2$ are non negative. Furthermore, $g_2$ has the same integrability property as $h_2$: it is integrable at $- \infty$ and is not at $+ \infty$. Indeed, there holds
    \begin{align*}
        \int_{- \infty}^{0} g_2 (x) \diff x &= \int_{-\infty}^{0} \int_{- \infty}^{+ \infty} h_2 (x-y) \, h_1 (y) \diff y \diff x \\
            &= \int_{- \infty}^0 \int_{- \infty}^{x} h_2 (x-y) \, h_1 (y) \diff y \diff x \\
            &= \int_{- \infty}^0 \int_{y}^0 h_2 (x-y) \, h_1 (y) \diff x \diff y \\
            &= \int_{- \infty}^0 h_1 (y) \int_{y}^0 h_2 (x-y) \diff x \diff y \\
            &= \int_{- \infty}^0 h_1 (y) \int_{0}^{- y} h_2 (x) \diff x \diff y \\
            &= \int_{- \infty}^0 h_1 (y) \ln{(1-y)} \diff x \diff y < \infty.
    \end{align*}
    However, it is still not integrable:
    \begin{equation*}
        \int_{-\infty}^{+\infty} g_2 (x) \diff x = \int_{-\infty}^{+\infty} h_1 (x) \diff x \, \int_{-\infty}^{+\infty} h_2 (y) \diff y = + \infty.
    \end{equation*}
    Hence, we can define the following function:
    \begin{equation*}
        f_2 (x) \coloneqq \int_{-\infty}^x g_2 (y) \diff y.
    \end{equation*}
    By definition, $f_2' = g_2 \in \mathcal{H}_\delta^0$ for all $\delta > 0$, $\lim_{- \infty} f = 0$ and $\lim_{+ \infty} f = + \infty$. Hence, $a_- + f_2$ satisfies the needed properties.
    We can recover the other cases by adding a $-$ and/or considering $f_2 (-x)$.
    
    \textit{Third step.} If both $a_+$ and $a_-$ are $\pm \infty$, one can consider $f_3 (x) = \pm f_2 (x) \pm f_2 (-x)$.
\end{proof}

\begin{rem}
    In particular, there also holds $f' \in \mathcal{H}^\ell_\delta$ for any $\delta, \ell > 0$.
\end{rem}

\begin{rem}
    Even though this result is in dimension $1$ for simplicity, the previous constructions can be extended to higher dimensions.
\end{rem}

\bibliographystyle{abbrv}
\bibliography{sample}

\end{document}